\documentclass{amsart}
\usepackage{amsmath}
\usepackage{amsfonts}
\usepackage{amssymb}
\usepackage{epsfig}
\newtheorem{theorem}{Theorem}[section]

\newtheorem{lm}[theorem]{Lemma}
\newtheorem{tr}[theorem]{Theorem}
\newtheorem{cor}[theorem]{Corollary}

\newtheorem{rem}[theorem]{Remark}
\newtheorem{pr}[theorem]{Proposition}

\newtheorem{hyp}[theorem]{Hypothesis}
\usepackage{color}

\begin{document}
	
	\title{Notes on the Duflo-Serganova functor in positive characteristic}
	\author{}
	\address{}
	\email{}
	\author{A. N. Zubkov}
	\address{Sobolev Institute of Mathematics, Omsk Branch, Pevtzova 13, 644043 Omsk, Russian Federation}
	\email{a.zubkov@yahoo.com}
	\begin{abstract}
	We develop a fragment of the theory of Duflo-Serganova functor over a field of odd characteristic. We elaborate a method of computing the symmetry supergroup $\widetilde{\mathbb{G}_x}$ of this functor, recently introduced by A.Sherman, for a wide class of supergroups $\mathbb{G}$, and apply it to the case when $\mathbb{G}$ is $\mathrm{GL}(m|n)$ or $\mathrm{Q}(n)$, and a square zero odd element $x\in \mathrm{Lie}(\mathbb{G})$ has minimal or maximal rank. For any quasi-reductive supergroup $\mathbb{G}$, which has a pair of specific parabolic supersubgroups, we prove a criterion of injectivity of a $\mathbb{G}$-supermodule, involving vanishing of Duflo-Serganova functor on it.    
	\end{abstract}
	\maketitle
	
	\section*{introduction}
	
These notes were inspired by the recent publication \cite{asher} and by the survey \cite{DSfunctor}. Namely, in \cite{asher} the symmetry supergroup of Duflo-Serganova functor was
introduced over $\mathbb{C}$. Let us briefly recall that if $M$ is a supermodule over Lie superalgebra $\mathfrak{g}$ and $x\in\mathfrak{g}_1$ satisfies $[x, x]=0$, then 
$x$ induces a square zero odd endomorphism $x_M$ of superspace $M$. The cohomology $M_x=\ker(x_M)/\mathrm{Im}(x_M)$ has the natural supermodule structure over the Lie superalgebra 
$\mathfrak{g}_x=\mathrm{Cent}_{\mathfrak{g}}(x)/[\mathfrak{g}, x]$. In other words, we have the functor $M\mapsto M_x$ from the category of $\mathfrak{g}$-supermodules to the category of
$\mathfrak{g}_x$-supermodules, called the \emph{Duflo-Serganova} functor (shortly, DS-functor). 

If $\mathfrak{g}$ is the Lie superalgebra of an algebraic supergroup $\mathbb{G}$, then $\mathfrak{g}_x$-supermodule structure of $M_x$ can be naturally associated with the $\mathbb{G}_x$-supermodule structure, where  $\mathbb{G}_x=\mathrm{Cent}_{\mathbb{G}}(x)/\mathbb{H}$ and $\mathbb{H}$ is the smallest normal supersubgroup of $\mathrm{Cent}_{\mathbb{G}}(x)$, such that $[\mathfrak{g}, x]\subseteq \mathrm{Lie}(\mathbb{H})$. The existence of $\mathbb{H}$ is easily proved using the Harish-Chandra pair method.
Moreover, $[\mathfrak{g}, x]_1=[\mathfrak{g}_0, x]=\mathrm{Lie}(\mathbb{H})_1$, but
$\mathrm{Lie}(\mathbb{H})_0$ is not always equal to $[\mathfrak{g}, x]_0=[\mathfrak{g}_1, x]$. Nevertheless, in many cases $\mathrm{Lie}(\mathbb{G}_x)=\mathfrak{g}_x$ and $\mathbb{G}_x$ can be regarded as a \emph{minimal} symmetry supergroup of DS-functor. 

A.Sherman found that the minimal symmetry supergroup can be expanded by some new odd symmetries (see \cite[Theorem 1.2]{asher}). This larger symmetry supergroup is denoted by $\widetilde{\mathbb{G}_x}$. Recall that the properties of DS-functor have been studied for a long time and its importance for representation theory is beyond doubt (see \cite{DSfunctor} for better understanding of the present state of affairs). The emergence of new symmetries may reveal more about these properties. In addition, a natural question arises whether the already developed theory of DS-functor in zero characteristic is transferred to the case of a field of positive characteristic? In these notes we initiate the study of some properties of DS-functor in arbitrary characteristic not equal to two. 

First, over a field of positive characteristic we develop a method of computing $\widetilde{\mathbb{G}_x}$ for a wide class of supergroups $\mathbb{G}$, such that $\Bbbk[\mathbb{G}]$ is a localization of its supersubbialgebra $B$ at a group like element, and $B$ is freely generated by a finite dimensional supersubcoalgebra. More precisely, the computing of Hopf superalgebra $\Bbbk[\widetilde{\mathbb{G}_x}]$ can be reduced to calculating the cohomology of the de Rham complex related to $B$, which is done using a kind of \emph{Cartier isomorphism} (see Lemma \ref{exactness of p-restricted de Rham} and Remark \ref{Cartier isomorphism}). Then we describe the coproduct of $\Bbbk[\widetilde{\mathbb{G}_x}]$ on its specific generators and apply the obtained formulas to the same supergroups as in \cite[Theorem 1.2]{asher} and particular $x$. It is long known that the cohomology of the de Rham complex in positive characteristic is much larger than in the case of zero characteristic. The same phenomenon occurs here. For example, if $\mathbb{G}\simeq\mathrm{GL}(m|n)$ and $x$ has rank one, then  $\widetilde{\mathbb{G}_x}\simeq \mathrm{GL}(m-1|n-1)\times \mathbb{G}^-_a$, provided $char\Bbbk =0$. But if $char\Bbbk>0$, then $\widetilde{\mathbb{G}_x}\simeq\mathrm{GL}(m-1|n-1)\ltimes\mathbb{R}$, where $\mathbb{R}\simeq \mathbb{G}_m\ltimes\mathbb{R}_u$ and the unipotent radical $\mathbb{R}_u$ has a central series with quotients $\mathbb{G}_a^-\times (\mathbb{G}_a)^{m+n-2}$ and $(\mathbb{G}_a^-)^{m+n-2}$. Furthermore, the new symmetries from $\mathbb{R}$ are relevant. In the matter of fact, we show that there is a finite dimensional $\mathbb{G}$-supermodule $W$, such that the induced supergroup morphism $\widetilde{\mathbb{G}_x}\to\mathrm{GL}(W_x)$ is injective. 	 

Another topic we consider here is finding out, using the DS-functor, whether a given supermodule $M$ is injective. Unlike the case of zero characteristic, the condition that $M$ is injective being restricted to any purely odd \emph{root supersubgroup} (this is equivalent to the condition $M_x=0$ for any square zero $x\in\mathfrak{g}_1$), is necessary but not sufficient even if $\mathbb{G}$ is purely odd. Also, injective $\mathbb{G}$-supermodules are almost never finite dimensional. Nevertheless, if $\mathbb{G}$ is \emph{quasi-reductive} and has so-called \emph{distinguished parabolic} supersubgroup $\mathbb{P}^-$ and $\mathbb{P}^+$,
then some criterion of injectivity, involving vanishing of DS-functor on $M$, is obtained (see Theorem \ref{char free Theorem 10.4} and Corollary \ref{one more about Thorem 10.4} for more detail). For example, the supergroups $\mathrm{GL}(m|n), \mathrm{SL}(m|n), m\neq n,$ and $\mathrm{P}(n)$ satisfy these conditions.

\section{Definitions and notations}

For the content of this section we refer to \cite{mas1, maszub, zub5}.

Let $\Bbbk$ be a field of odd or zero characteristic. A $\mathbb{Z}_2$-graded vector space $V=V_0\oplus V_1$ is called a vector \emph{superspace}. We have the \emph{parity function}
$V_0\sqcup V_1\setminus 0\to\mathbb{Z}_2$, determined as $v\mapsto |v|$, if $v\in V_{|v|}$. The category of vector superspaces $\mathsf{SV}_{\Bbbk}$ is a full subcategory of the category of vector spaces. Note that $\mathsf{SV}_{\Bbbk}$ is not abelian, but the subcategory of vector superspaces with only graded linear maps between them, is. The latter is denoted by $\mathsf{Sv}_{\Bbbk}$. Both are \emph{symmetric monoidal} categories with respect to the \emph{braiding} :
\[V\otimes W\to W\otimes V, v\otimes w\mapsto (-1)^{|v||w|}w\otimes v, v\in V, w\in W,\]
and $\mathsf{SV}_{\Bbbk}$ admits an endofunctor $\Pi$, that is called the \emph{parity change functor}, such that $(\Pi V)_{a}=V_{a+1\pmod{2}}$ for any superspace $V$.

Algebra, coalgebra, bialgebra and Hopf algebra objects in $\mathsf{Sv}_{\Bbbk}$ are called \emph{superalgebra, supercoalgebra, superbialgebra and Hopf superalgebra} respectively.
If $A$ is a superalgebra, then the full subcategory of the category of left $A$-modules , consisting of $\mathbb{Z}_2$-graded $A$-modules, is denoted by
$_{A}\mathsf{SMod}$. This category is not abelian, but the category $_{A}\mathsf{Smod}$ 
consisting of the same objects with graded morphisms only, is. Finally, let $_{A}\underline{\mathsf{SMod}}$ denote the category of left
$A$-supermodule with ungraded morphisms $\phi : M\to N$, such that 
\[a\phi(m)=(-1)^{|a||\phi|}\phi(am), a\in A, m\in M. \] The right hand side versions of these definitions are obvious.

Similarly, if $A$ is a supercoalgebra, then the full subcategory of the category of right $A$-comodules, consisting of $\mathbb{Z}_2$-graded $A$-comodules, is denoted by $\mathsf{SMod}^A$. This category is not abelian again, but the category $\mathsf{Smod}^A$,
consisting of the same objects with graded morphisms only, is. 

A superalgebra $A$ is said to be \emph{super-commutative}, if the homogeneous elements of $A$ satisfy the identity $ab=(-1)^{|a||b|}ba$. The category of super-commutative superalgebras with graded morphisms is denoted by $\mathsf{SAlg}_{\Bbbk}$, and it is anti-equivalent to the category $\mathsf{SSch}_{\Bbbk}$ of \emph{affine superschemes} over $\Bbbk$. More precisely, each $A\in\mathsf{SAlg}_{\Bbbk}$ determines the superscheme $\mathbb{X}=\mathrm{SSp}(A)$, such that  
\[\mathbb{X}(B)=\mathrm{Hom}_{\mathsf{SAlg}_{\Bbbk}}(A, B), B\in\mathsf{SAlg}_{\Bbbk}.\]
An affine superscheme 
$\mathrm{SSp}(A)$ is called \emph{algebraic}, if its \emph{coordinate superalgebra} $A$ is finitely generated. A \emph{closed supersubscheme} $\mathbb{Y}$ of $\mathbb{X}$ is a subfunctor
of $\mathbb{X}$, such that
\[\mathbb{Y}(B)=\{x\in\mathbb{X}(B)\mid x(I)=0 \}, \ B\in\mathsf{SAlg}_{\Bbbk}, \]
where $I$ is a superideal of $A$. The superideal $I$ is called the \emph{defining} superideal of $\mathbb{Y}$. Note that $\mathbb{Y}\simeq \mathrm{SSp}(A/I)$. For example, set $I=I_A=AA_1$. Then $I_A$ determines the largest
\emph{purely even} supersubscheme $\mathbb{X}_{ev}\simeq\mathrm{SSp}(A/I_A)\simeq\mathrm{SSp}(A_0/A_1^2)$, that can be regarded as an ordinary affine scheme. The latter is denoted by $\mathbb{X}_{res}$.

An \emph{algebraic supergroup} $\mathbb{G}$ is a group functor, that is an algebraic superscheme as well. Equivalently, $\mathbb{G}\simeq \mathrm{SSp}(A)$, where $A$ is a finitely generated super-commutative Hopf superalgebra. From now on the Hopf superalgebra $A$ is denoted by $\Bbbk[\mathbb{G}]$. A group subfunctor $\mathbb{H}$ of $\mathbb{G}$
is a closed supersubgroup if and only if its defining superideal is a \emph{Hopf superideal}. For example, $I_{\Bbbk[\mathbb{G}]}$ determines the largest \emph{purely even} supersubgroup
$\mathbb{G}_{ev}$.  

The category of left $\mathbb{G}$-supermodules with ungraded morphisms is identified with $\mathsf{SMod}^{\Bbbk[\mathbb{G}]}$. It is denoted by $_{\mathbb{G}}\mathsf{SMod}$. As above, we also set $_{\mathbb{G}}\mathsf{Smod}\simeq \mathsf{Smod}^{\Bbbk[\mathbb{G}]}$.  Recall that if a left $\mathbb{G}$-supermodule is injective in $\mathsf{Smod}^{\Bbbk[\mathbb{G}]}$, then it is injective in $\mathsf{SMod}^{\Bbbk[\mathbb{G}]}$ (see \cite[Proposition 3.1]{zub5}). The right derived functors of $\mathrm{Hom}_{_{\mathbb{G}}\mathsf{SMod}}(M, \ )$ and
$\mathrm{Hom}_{_{\mathbb{G}}\mathsf{Smod}}(M, \ )$ are denoted by $\mathrm{Ext}^{\bullet}_{_{\mathbb{G}}\mathsf{SMod}}(M, \ )$ and $\mathrm{Ext}^{\bullet}_{_{\mathbb{G}}\mathsf{Smod}}(M, \ )$ respectively. As it has been observed in \cite{zub5}, there is
\[\mathrm{Ext}^{\bullet}_{_{\mathbb{G}}\mathsf{SMod}}(\Pi^a M, \Pi^b N)\simeq \Pi^{a+b}\mathrm{Ext}^{\bullet}_{_{\mathbb{G}}\mathsf{SMod}}(M, N)\]
for any $\mathbb{G}$-supermodules $M$ and $N$, and any $a, b=0, 1$. In particular, we have  
\[\mathrm{Ext}^{\bullet}_{_{\mathbb{G}}\mathsf{SMod}}(M, N)_0\simeq \mathrm{Ext}^{\bullet}_{_{\mathbb{G}}\mathsf{Smod}}(M, N)\]
and
\[\mathrm{Ext}^{\bullet}_{_{\mathbb{G}}\mathsf{SMod}}(M, N)_1\simeq \mathrm{Ext}^{\bullet}_{_{\mathbb{G}}\mathsf{Smod}}(\Pi M, N)\simeq \mathrm{Ext}^{\bullet}_{_{\mathbb{G}}\mathsf{Smod}}(M, \Pi N).\]

In what follows, a morphism of super-objects is always understood as graded morphism, unless stated otherwise.

	\section{Purely odd supergroups}
	An algebraic supergroup $\mathbb{G}$ is said to be \emph{purely odd}, if $\mathbb{G}_{ev}=1$. This condition is equivalent to $\mathbb{G}\simeq (\mathbb{G}_a^-)^r$, where
	$\mathbb{G}^-_a$ is the $0|1$-dimensional purely odd unipotent supergroup (cf. \cite[Proposition 11.1]{maszub}). More precisely, the Hopf superalgebra $\Bbbk[\mathbb{G}^-_a]$ is isomorphic to $\Bbbk[z]$, where the (free) generator $z$ is odd and primitive. Therefore, if $\mathbb{G}$ is a purely odd supergroup, then $\Bbbk[\mathbb{G}]\simeq\Lambda(\mathfrak{g}^*)$, where all elements of $\mathfrak{g}^*$ are primitive.
	
	Let $\mathbb{G}$ be a purely odd supergroup. The distribution superalgebra $\mathrm{Dist}(\mathbb{G})$ of $\mathbb{G}$ is isomorphic to $\Lambda(\mathfrak{g})$. By \cite[Lemma 1.3]{zub3} the category $_{\mathbb{G}}\mathsf{Smod}$ is equivalent to the category $_{\mathrm{Dist}(\mathbb{G})}\mathsf{Smod}$. 
	\begin{rem}
	Note that the formulation of \cite[Lemma 1.3]{zub3} is not entirely correct. More precisely, if $R$ is a finite dimensional Hopf superalgebra, then it is true that $\mathsf{Smod}^R\simeq _{R^*}\mathsf{Smod}$. But it does not take place for larger categories. The point is that the functor, that is proposed there as an equivalence, maps 
	$\mathsf{SMod}^R$ to $_{R^*}\underline{\mathsf{SMod}}$, not to $_{R^*}\mathsf{SMod}$.
	\end{rem}
	By \cite[Lemma 4.1 and Lemma 4.4]{zub3} a $\mathrm{Dist}(\mathbb{G})$-supermodule is injective if and only if it is projective if and only if it is free.
	\begin{lm}\label{free basis}
		Let $\mathbb{G}$ be a purely odd supergroup. Then any $\mathbb{G}$-supermodule $M$ contains a maximal free supersubmodule $F$, that is a direct summand of $M$. 	
	\end{lm}
	\begin{proof}
		Choose a homogeneous basis $z_1, \ldots, z_r$ of $\mathfrak{g}$. If $M$ is a $\mathrm{Dist}(\mathbb{G})$-supermodule, choose a homogeneous basis $\{(z_1\ldots z_r )m_i\}_{i\in I}$
		of the superspace $(z_1\ldots z_r)M$. The elements $m_i$ generate a free $\mathrm{Dist}(\mathbb{G})$-supersubmodule $F$ of $M$. Since $F$ is also injective, we have
		$M=F\oplus G$, where $(z_1\ldots z_r)G=0$. In other words, $G$ or $M/F$ does not contain nontrivial free supersubmodules.
	\end{proof}	
	\begin{rem}\label{extension invariant}
		A $\mathbb{G}$-supermodule $M$ is free if and only if $M\otimes L$ is free as a $\mathrm{Dist}(\mathbb{G})\otimes L$-supermodule for any field extension $\Bbbk\subseteq L$.	
	\end{rem}
	\section{Harish-Chandra pairs}
	
	A pair $(G, V)$, where $G$ is an algebraic group, and $V$ is a $G$-module, is a \emph{Harish-Chandra pair} if the following conditions hold:
	\begin{enumerate}
		\item There is a symmetric bilinear map $V\times V\to\mathrm{Lie}(G)$, denoted by $[ \ , \ ]$;
		\item This map is $G$-equivariant w.r.t. the diagonal action of $G$ on $V\times V$ and the adjoint action of $G$ on $\mathrm{Lie}(G)$;
		\item The induced action of $\mathrm{Lie}(G)$ on $V$, denoted by the same symbol $[ \ , \ ]$,  satisfies $[[v, v], v]=0$. 
	\end{enumerate}
	The morphism of Harish-Candra pairs $(G, V)\to (H, W)$ is a couple of morphisms $\mathsf{f} : G\to H$ and $\mathsf{u} : V\to W$ such that 
	\begin{enumerate}
		\item $\mathsf{u} (gv)= \mathsf{f}(g)\mathsf{u}(v), g\in G, v\in V$;
		\item $[\mathsf{u}(v), \mathsf{u}(v')]=\mathrm{d}_e(\mathsf{f})([v, v']), v, v'\in V$. 
	\end{enumerate}  
	We have a natural functor from the category of algebraic supergroups to the category of Harish-Chandra pairs, determined as
	\[\Phi(\mathbb{G})=(G, \mathrm{Lie}(\mathbb{G})_1), G=\mathbb{G}_{res}.\]	
	Here $\mathfrak{g}=\mathrm{Lie}(\mathbb{G})$ is regarded as a $\mathbb{G}$-supermodule with respect to the adjoint action. 
	\begin{theorem}\label{equivalence}(cf. \cite[Theorem 12.10]{maszub})
		The functor $\Phi$ is an equivalence.	
	\end{theorem}	
	A sequence $(R, W) \to (G, V) \to (H, U)$ in the category of Harish-Chandra pairs is called \emph{exact} whenever the following conditions hold:	
	\begin{enumerate}
		\item[(s1)] The sequences $R \to G \to H$ and $W \to  V \to U$ are exact in the categories of algebraic groups  and vector spaces respectively;
		\item[(s2)] $W$ is a $G$-submodule of $V$;
		\item[(s3)] $R$ acts trivially on $V/W\simeq U$;
		\item[(s4)] $[V, W]\subseteq\mathrm{Lie}(R)$.
	\end{enumerate}	
	\begin{theorem}\label{short sequences}(cf. \cite[Theorem 12.11]{maszub})
		The short sequence $\mathbb{R}\to\mathbb{G}\to\mathbb{H}$ is exact in the category of algebraic supergroups if and only if the sequence
		$\Phi(\mathbb{R})\to\Phi(\mathbb{G})\to\Phi(\mathbb{H})$ is exact in the category of Harish-Chandra pairs. 		
	\end{theorem}
	In particular, if an algebraic supergroup $\mathbb{G}$ is represented by a Harish-Chandra pair $(G, V)$, then a subpair $(R, W)$ corresponds to a normal supersubgroup if and only if the latter satisfies the conditions (s2)-(s4).  
	
	Let $\mathbb{G}$ be an algebraic supergroup, represented by a Harish-Chandra pair $(G, \mathfrak{g}_1)$. 
	\begin{lm}\label{smallest associated with superideal}
		Let $\mathfrak{h}$ be a superideal of $\mathfrak{g}$, such that $\mathfrak{h}_1$ is a $G$-submodule of $\mathfrak{g}_1$. Set $H=\ker(G\to\mathrm{GL}(\mathfrak{g}_1/\mathfrak{h}_1))$.
		Then the Harish-Chandra subpair $(H, \mathfrak{h}_1)$ represents a normal supersubgroup $\mathbb{H}$ of $\mathbb{G}$, such that $\mathrm{Lie}(\mathbb{H})$ contains $\mathfrak{h}$ and $\mathrm{Lie}(\mathbb{H})_1=\mathfrak{h}_1$. In particular, if $\mathbb{R}$ is the smallest normal supersubgroup of $\mathbb{G}$, such that  $\mathfrak{h}\subseteq \mathrm{Lie}(\mathbb{R})$, then
		$\mathfrak{h}_1= \mathrm{Lie}(\mathbb{R})_1$.	
	\end{lm}	
	\begin{proof}
		We have $\mathrm{Lie}(H)=\{y\in\mathfrak{g}_0 \mid [y, \mathfrak{g}_1]\subseteq \mathfrak{h}_1 \}$. Since $\mathfrak{h}$ is a superideal, 
		thus immediately follows $[\mathfrak{h}_1, \mathfrak{g}_1]\subseteq \mathrm{Lie}(H)$ and $\mathfrak{h}_0\subseteq \mathrm{Lie}(H)$. 	
	\end{proof}
	
	\section{Duflo-Serganova functor for supergroups}
	
	Let $V$ be a vector space and $x$ be a square-zero endomorphism of $V$. Then $\mathrm{Im}(x)\subseteq\ker(x)$ and the factorspace $\ker(x)/\mathrm{Im}(x)$ is denoted by $V_x$. 
	
	Let $\mathbb{G}$ be an algebraic supergroup. Set $\mathfrak{g}=\mathrm{Lie}(\mathbb{G})$.
	
	Let $x\in \mathfrak{g}_1$, such that $[x, x]=0$. If $M$ is a $\mathfrak{g}$-supermodule, then the element $x$ induces a square-zero odd endomorphism $x_M$ of the superspace $M$, hence one can define a superspace $M_x=\ker(x_M)/\mathrm{Im}(x_M)$. In particular, regarding $\mathfrak{g}$ as a $\mathfrak{g}$-supermodule with respect to the adjoint action, one sees that $\mathfrak{g}_x=\mathrm{Cent}_{\mathfrak{g}}(x)/[\mathfrak{g}, x]$ is a Lie superalgebra. Indeed, it is obvious that  $\mathrm{Cent}_{\mathfrak{g}}(x)$ is a Lie supersubalgebra of $\mathfrak{g}$ and 	$[\mathfrak{g}, x]$ is its Lie superideal. Furthermore, $M_x$ has a natural structure of $\mathfrak{g}_x$-supermodule, so that we obtain a functor from the category
	of $\mathfrak{g}$-supermodules to the category of $\mathfrak{g}_x$-supermodules.  This functor is said to be the \emph{Duflo-Serganova functor} (cf. \cite{DSfunctor}).
	
	It is clear that the closed supersubgroup $\mathrm{Cent}_{\mathbb{G}}(x)$ is represeneted by the Harish-Chandra subpair $(\mathrm{Cent}_G(x), \mathrm{Cent}_{\mathfrak{g}}(x)_1)$. 
	Note that $[\mathfrak{g}_0, x]$ is an $\mathrm{Cent}_G(x)$-submodule of $\mathrm{Cent}_{\mathfrak{g}}(x)_1$. Then Lemma \ref{smallest associated with superideal} implies that 
	there is a smallest normal supersubgroup $\mathbb{R}$ of $\mathrm{Cent}_{\mathbb{G}}(x)$, such that $\mathrm{Lie}(\mathbb{R})$ contains $[\mathfrak{g}, x]$. Let $\mathbb{G}_x$ denote the supergroup $\mathrm{Cent}_{\mathbb{G}}(x)/\mathbb{R}$.	
	\begin{lm}\label{DS functor for supergroups}
		Let $M$ be a $\mathbb{G}$-supermodule. Then $M_x$ has a natural structure of $\mathbb{G}_x$-supermodule. Moreover, $M\mapsto M_x$ is a functor from
		$_{\mathbb{G}}\mathsf{Smod}$ to $_{\mathbb{G}_x}\mathsf{Smod}$.	
	\end{lm}
	\begin{proof}
		It is clear that $\mathrm{Cent}_{\mathbb{G}}(x)$ stabilizes both $\ker(x_M)$ and $\mathrm{Im}(x_M)$, hence $M_x$ has a natural structure of $\mathrm{Cent}_{\mathbb{G}}(x)$-supermodule.
		Let $\mathbb{H}$ denote the closed normal supersubgroup $\mathrm{Cent}_{\mathrm{Cent}_{\mathbb{G}}(x)}(M_x)$. Then \[\mathrm{Lie}(\mathbb{H})=\{y\in\mathrm{Cent}_{\mathfrak{g}}(x)\mid y(\ker(x_M))\subseteq\mathrm{Im}(x_M) \}\]
		contains $[\mathfrak{g}, x]$, whence $\mathbb{R}\leq\mathbb{H}$.  The functoriality is obvious. 	
	\end{proof}
	Following \cite{DSfunctor}, we call again this functor a \emph{Duflo-Serganova functor} or shortly, \emph{DS-functor}. The proof of the following elementary properties of DS-functor can be copied from \cite[Lemma 2.4]{DSfunctor}. Nevertheless, we give some details below.
	\begin{lm}\label{elementary prop-s}
		Let $M, N$ be $\mathbb{G}$-supermodules.	
		\begin{enumerate}
			\item There is a natural isomorphism $M_x\otimes N_x\simeq (M\otimes N)_x$ of $\mathbb{G}_x$-supermodules.
			\item There is a natural isomorphism $(M^*)_x\simeq (M_x)^*$ of $\mathbb{G}_x$-supermodules, provided $M$ is finite dimensional.  
			\item There is a natural isomorphism $(M\oplus N)_x\simeq M_x\oplus N_x$ of $\mathbb{G}_x$-supermodules.
		\end{enumerate}	
	\end{lm} 
	\begin{proof}
		Let $\mathbb{U}\simeq \mathbb{G}_a^-$ denote the purely odd unipotent subgroup of $\mathbb{G}$, that corresponds to the Harish-Chandra subpair $(1, \Bbbk x)$. Thus 
		$M$ can be regarded as a $\mathbb{U}$-supermodule, hence $\mathrm{Dist}(\mathbb{U})\simeq\Bbbk[x]$-supermodule via the natural superalgebra morphism $\mathrm{Dist}(\mathbb{U})
		\to\mathrm{Dist}(\mathbb{G})\to\mathrm{End}_{\Bbbk}(M)$, that takes $x$ to $x_M$.
		Lemma \ref{free basis} implies that there is a free $\Bbbk[x]$-supersubmodule $F$, such that $M=S\oplus F$, where $S\simeq M_x$. 
		The concluding arguments are the same as in \cite[Lemma 2.4(1)]{DSfunctor}.
	
	Note that if $M$ is finite dimensional, then the rank of free supersubmodule $F$ is equal to the codimension of the subspace $\ker(x_M)$ in $M$ (as well as to the dimension of $x_MM$).
		
		Recall that the element $x$ acts on $M^*$ as $(x\phi)(m)=(-1)^{|\phi|}\phi(xm), \phi\in M^*, m\in M$. Following Lemma \ref{free basis}, we fix a decomposition
		$M=M_x\oplus U\oplus xU$, where a basis of $U$ is the basis of the largest free $\Bbbk[x]$-supersubmodule $F$ of $M$. 
		It is easy to see that $x_{M^*}M^*\subseteq U^*$ and $\ker(x_{M^*})=M_x^*\oplus U^*$. The above remark implies that
		the largest free $\Bbbk[x]$-supersubmodule of $M^*$ is freely generated by $(xU)^*$, hence $(M^*)_x\simeq M_x^*$ and $x_{M^*}(xU)^*=U^*$.  
		
		The statement (3) is trivial.
	\end{proof}
	\begin{cor}\label{for complexes}
		Let $\mathrm{K}=\{\mathrm{K}_p\}$ and $\mathrm{L}=\{\mathrm{L}_q\}$ be finite (cochain) complexes. Then $\mathrm{H}^{\bullet}(\mathrm{K}\otimes\mathrm{L})\simeq \mathrm{H}^{\bullet}(\mathrm{K})\otimes \mathrm{H}^{\bullet}(\mathrm{L})$. 	
	\end{cor}
	\begin{proof}
		Indeed, let $x$ denote the differentials of both complexes. Then $x$ acts as an odd endomorphism on the superspaces ${\bf K}=\oplus_{p}\mathrm{K}_p$ and
		${\bf L}=\oplus_q\mathrm{L}_q$, where ${\bf K}_i=\oplus_{p\equiv i\pmod{2}}\mathrm{K}_p, {\bf L}_i=\oplus_{p\equiv i\pmod{2}}\mathrm{L}_q, i=0, 1$. It remains to note that
		${\bf K}_x=\mathrm{H}^{\bullet}(\mathrm{K}), {\bf L}_x=\mathrm{H}^{\bullet}(\mathrm{L})$ and apply Lemma \ref{elementary prop-s}. 	
	\end{proof}
	\begin{rem}\label{graded complement}
		Note that $\mathrm{H}^{\bullet}(\mathrm{K}\otimes\mathrm{L})\simeq \mathrm{H}^{\bullet}(\mathrm{K})\otimes \mathrm{H}^{\bullet}(\mathrm{L})$ is not necessary an isomorphism of $\mathbb{N}$-graded spaces. But if there are $\mathbb{N}$-graded complements to the free $\Bbbk[x]$-supersubmodules in ${\bf K}$ and ${\bf L}$ respectively, then it is.  	
	\end{rem}
	\begin{rem}\label{identification}
		In what follows we identify a finite complex $\mathrm{K}$ with the associated $\mathbb{N}$-graded space ${\bf K}$, if it does not lead to confusion.	
	\end{rem}
	From now on, if $f : M\to N$ is a morphism of $\mathbb{G}$-supermodules, then $f_x$ denotes the induced morphism $M_x\to N_x$ of $\mathbb{G}_x$-supermodules. 
	Following \cite{DSfunctor}, we also denote $\{x\in\mathfrak{g}_1\mid [x, x]=0 \}$ by $X$, and $\{x\in X\mid M_x\neq 0\}$ by $X_M$ for any $\mathbb{G}$-supemodule $M$.
	
	\section{$\mathbb{G}$-superschemes and DS-functor}
	
	Let $\mathbb{X}\simeq\mathrm{SSp}(A)$ be an affine superscheme on which $\mathbb{G}$ acts on the right. We call it a \emph{right $\mathbb{G}$-superscheme}. This is equivalent to say that $A$ is a right $\Bbbk[\mathbb{G}]$-supercomodule, hence a left $\mathbb{G}$-supermodule, such that the comodule map $\tau_A : A\to A\otimes\Bbbk[\mathbb{G}]$ is a superalgebra morphism. We have the induced left action of superalgebra $\mathrm{Dist}(\mathbb{G})$ on $A$, determined as
	\[\phi\cdot a=\sum (-1)^{|\phi||a_{(1)}|}a_{(1)}\phi(a_{(2)}), \phi\in\mathrm{Dist}(\mathbb{G}), a\in A , \tau_A(a)=\sum a_{(1)}\otimes a_{(2)}. \]
	This action satisfies the identity :
	\[\phi\cdot (ab)=\sum (-1)^{|\phi_{(2)}||a|}(\phi_{(1)}\cdot a)(\phi_{(2)}\cdot b), a, b\in A, \ \Delta_{\mathrm{Dist}(\mathbb{G})}(\phi)=\sum\phi_{(1)}\otimes\phi_{(2)}. \]
	In particular, any (homogeneous) element of $\mathfrak{g}=\mathrm{Lie}(\mathbb{G})$ acts on $A$ as a superderivation. 
	
	We also need the following fact : if $M$ and $N$ are $\mathbb{G}$-supermodules, then the $\mathrm{Dist}(\mathbb{G})$-supermodule structure of $M\otimes N$ is determined as
	\[\phi\cdot(m\otimes n)=\sum (-1)^{|\phi_{(2)}||m|}\phi_{(1)}\cdot m\otimes \phi_{(2)}\cdot n, m\in M, n\in N, \phi\in\mathrm{Dist}(\mathbb{G}).\]
	The following is a folklore.
	\begin{rem}\label{a folklore fact}
		Let $V$ be a right $B$-supercomodule over a Hopf superalgebra $B$. Let $I$ be a Hopf superideal of $B$. Set $C=B/I$. If $V$ is a trivial $C$-supercomodule, then 
		$\tau_V(V)\subseteq V\otimes B^C$. Just note that the comodule map $\tau_V$ is a supercomodule morphism $V\to V_{triv}\otimes B=B^{\oplus \dim V_0}\oplus \Pi B^{\oplus \dim V_1}$, where $B$ is regarded as a right $B$-supercomodule with respect to its comultiplication map.  	
	\end{rem} 
Let $x\in\mathfrak{g}_1$ and $[x, x]=0$.
	\begin{lm}\label{trivial}
		The superspace $A_x$ has the natural structure of a superalgebra. Moreover, the comodule map $A_x\to A_x\otimes\Bbbk[\mathbb{G}_x]$ is a superalgebra morphism. In other words,
		$\mathrm{SSp}(A_x)$ is a right $\mathbb{G}_x$-superscheme, which is denoted by $\mathbb{X}_x$.
	\end{lm} 
	\begin{proof}
		Since $x_A$ is a superderivation, then $\ker(x_A)$ is a supersubalgebra in $A$ and $\mathrm{Im}(x_A)$ is a superideal of $\ker(x_A)$. Moreover, the induced action of $\mathrm{Cent}_{\mathbb{G}}(x)$ on $A_x$ corresponds to the comodule map $A_x\to A_x\otimes\Bbbk[\mathrm{Cent}_{\mathbb{G}}(x)]$, that is obviously a superalgebra map. 
		Finally, if $\mathbb{R}\unlhd \mathrm{Cent}_{\mathbb{G}}(x)$ acts trivially on $A_x$, then by Remark \ref{a folklore fact}, $\tau_{A_x}(A_x)\subseteq A_x\otimes\Bbbk[\mathrm{Cent}_{\mathbb{G}}(x)]^{\Bbbk[\mathbb{R}]}\simeq A_x\otimes\Bbbk[\mathrm{Cent}_{\mathbb{G}}(x)/\mathbb{R}]$.  
	\end{proof}
	So, we have a version of Duflo-Serganova functor, that acts from the category of affine right $\mathbb{G}$-superschemes (with $\mathbb{G}$-equivariant morphisms between them) to the categroy 
	of affine right $\mathbb{G}_x$-superschemes.  The symmetric statement is valid for the category of affine left $\mathbb{G}$-superschemes. 
	\begin{rem}\label{DS for supermodules as a particular case}
		The above DS-functor for supermodules is a particular case of the DS-functor for superschemes. In the matter of fact, a superspace $M$ is a $\mathbb{G}$-supermodule if and only if
		$\mathrm{SSp}(\mathrm{Sym}(M))$ is a right $\mathbb{G}$-superscheme, such that the (super)comodule map $\tau_{\mathrm{Sym}(M)}$ preserves the natural $\mathbb{N}$-grading of
		$\mathrm{Sym}(M)$.  
	\end{rem}
	\begin{lm}
		Assume that $\mathrm{SSp}(A)$ is a left $\mathbb{H}$-superscheme and a right $\mathbb{G}$-superscheme, such that these two actions commute each to other. Then the superscheme $\mathrm{SSp}(A_x)$ is still a left $\mathbb{H}$-superscheme and a right $\mathbb{G}_x$-superscheme. Besides, these two actions commute each to other as well.	
	\end{lm}
	\begin{proof}
		We denote the left and right comodule maps of $\Bbbk[A]$ by $\tau_l$ and $\tau_r$ respectively. Then the identity
		\[(\tau_l\otimes\mathrm{id}_A)\tau_r=(\mathrm{id}_A\otimes\tau_r)\tau_l\]
		implies
		\[\tau_l(\phi\cdot f)=\sum (-1)^{|\phi||f_{(1)}|} f_{(1)}\otimes \phi\cdot f_{(2)}, \phi\in\mathrm{Dist}(\mathbb{G}), f\in A, \tau_l(f)=\sum f_{(1)}\otimes f_{(2)}.\]
		Without loss of generality, one can assume that the factors $f_{(1)}$ are linearly independent. Thus 
		\[\tau_l(\ker(x_A))\subseteq \Bbbk[\mathbb{H}]\otimes \ker(x_A) \ \mbox{and} \ \tau_l(\mathrm{Im}(x_A))\subseteq \Bbbk[\mathbb{H}]\otimes \mathrm{Im}(x_A).\]	
		Therefore, $A_x$ is a left $\Bbbk[\mathbb{H}]$-supercomodule and its comodule map is obviously a superalgebra map. The second statement is due Lemma \ref{trivial}.
	\end{proof}
	The supergroup $\mathbb{G}$ has a natural structure of a right $\mathbb{G}$-superscheme with respect to the \emph{right conjugation action} 
	\[(g, h)\mapsto h^{-1}gh, g, h\in\mathbb{G}(A), A\in\mathsf{SAlg}_{\Bbbk}. \]
	The corresponding (super)comodule map $\tau_{r-con}$ is given by
	\[f\mapsto \sum (-1)^{|f_{(1)}||f_{(2)}|} f_{(2)}\otimes S_{\mathbb{G}}(f_{(1)})f_{(3)}, f\in\Bbbk[\mathbb{G}], \]
	and in turn, $x$ acts as
	\[x\cdot f=\sum (-1)^{|f_{(1)}||f_{(2)}|+|f_{(2)}|} f_{(2)} x(S_{\mathbb{G}}(f_{(1)})f_{(3)})=\]
	\[\sum (-1)^{|f_{(1)}||f_{(2)}|+|f_{(2)}|} f_{(2)} x(S_{\mathbb{G}}(f_{(1)})) \epsilon_{\mathbb{G} }(f_{(3)})+\]\[\sum (-1)^{|f_{(1)}||f_{(2)}|+|f_{(2)}|+|f_{(1)}|} f_{(2)} \epsilon_{\mathbb{G}}(S_{\mathbb{G}}(f_{(1)})) x(f_{(3)})= \]
	\[-\sum x(f_{(1)})f_{(2)}+\sum (-1)^{|f_{(1)}|}f_{(1)}x(f_{(2)}),\]
	i.e. it is just the action from \cite[Section 4]{asher}. 
	
	Let $\mathbb{X}$ and $\mathbb{Y}$ be right $\mathbb{G}$-superschemes. Then $\mathbb{X}\times\mathbb{Y}$ can be regarded as a right $\mathbb{G}$-superscheme with respect to the diagonal action of $\mathbb{G}$. In particular, if $\mathbb{X}\times\mathbb{Y}\to\mathbb{Z}$ is a morphism of right $\mathbb{G}$-superschemes, then Lemma \ref{elementary prop-s} implies that there is a natural morphism $\mathbb{X}_x\times\mathbb{Y}_x\simeq(\mathbb{X}\times\mathbb{Y})_x\to\mathbb{Z}_x$ of $\mathbb{G}_x$-superschemes. Using this elementary remark we obtain the following.
	\begin{lm}\label{one more G_x}
		Let $\mathbb{G}$ be considered as a right $\mathbb{G}$-superscheme with respect to the right conjugation action. Then $\Bbbk[\mathbb{G}]_x$ has a natural structure of a Hopf superalgebra. In other words, $\mathrm{SSp}(\Bbbk[\mathbb{G}]_x)$ is an affine supergroup on which $\mathbb{G}_x$ acts by group automorphisms.
	\end{lm}
	\begin{proof}
		Just observe that all structural morphisms of $\mathbb{G}$, i.e. multiplication, inverse and unit, are morphisms of $\mathbb{G}$-superschemes with respect to the right conjugation action of $\mathbb{G}$ on itself.  	
	\end{proof}
	Following \cite{asher}, we denote this supergroup by $\widetilde{\mathbb{G}_x}$.
	\begin{lm}\label{M_x is again supermodule w.r.t another G_x}
		If $\mathbb{X}$ is a right $\mathbb{G}$-superscheme, then $\mathbb{X}_x$ is a right $\widetilde{\mathbb{G}_x}$-superscheme, so that we obtain another version of Duflo-Serganova functor, that acts	from the category of right $\mathbb{G}$-superschemes to the category of right $\widetilde{\mathbb{G}_x}$-superschemes.	In particular, if $M$ is a $\mathbb{G}$-supermodule, then $M_x$ has a natural structure of $\widetilde{\mathbb{G}_x}$-supermodule. 	
	\end{lm}
	\begin{proof}
		Consider the diagonal action of $\mathbb{G}$ on $\mathbb{X}\times\mathbb{G}$ by the rule :
		\[((x, g), h)\mapsto (xh, h^{-1}gh), x\in\mathbb{X}(A), g, h\in\mathbb{G}(A).\]
		Then the action morphism $\mathbb{X}\times\mathbb{G}\to\mathbb{X}$ is obviously a morphism of $\mathbb{G}$-superschemes, that induces the required action morphism
		$\mathbb{X}_x\times\widetilde{\mathbb{G}_x}\to\mathbb{X}_x$. The second statement follows by Remark \ref{DS for supermodules as a particular case}.
	\end{proof}
	\begin{rem}\label{Lemma 3.2} The statements of Lemma \ref{elementary prop-s} are still valid even if we replace $\mathbb{G}_x$ by $\widetilde{\mathbb{G}_x}$.
		
	\end{rem}
	The identity $(x\cdot f) (x\cdot g)=x\cdot (f (x\cdot g)), f, g\in\Bbbk[\mathbb{G}],$ implies that $A=\mathrm{Im}(x_{\Bbbk[\mathbb{G}] })$ is a supersubalgebra of $\Bbbk[\mathbb{G}]$
	such that $\epsilon_{\mathbb{G}}(A)=0$. We have $S_{\mathbb{G}}(x\cdot f)= x\cdot (S_{\mathbb{G}}(f))$ and 
	\[ \Delta_{\mathbb{G}}(x\cdot f)=\sum x\cdot f_{(1)}\otimes f_{(2)}+\sum (-1)^{|f_{(1)}|} f_{(1)}\otimes x\cdot f_{(2)}=x\cdot\Delta_{\mathbb{G}}(f),
	\]
	hence $I=\Bbbk[\mathbb{G}]A$ is a Hopf superideal, that defines a supersubgroup $\mathbb{L}=\mathrm{SSp}(\Bbbk[\mathbb{G}]/I)$ of $\mathbb{G}$.
	\begin{lm}\label{L is equal to Cent}
		There holds $\mathrm{Cent}_{\mathbb{G}}(x)=\mathbb{L}$.	
	\end{lm}
	\begin{proof}
		For any superalgebra $A$, let $A[\epsilon_0, \epsilon_1]$ denote the $A$-superalgebra of \emph{dual supernumbers} (cf. \cite{maszub}). Recall that $x$ can be identified with the superalgebra morphism $\Bbbk[\mathbb{G}]\to \Bbbk[\epsilon_0, \epsilon_1]$, that sends $f\in\Bbbk[\mathbb{G}]$ to $\epsilon_{\mathbb{G}}(f)+\epsilon_1 x(f)$. Slightly abusing notations, 
		we denote this morphism by the same symbol $x$.
		The adjoint action is determined as
		\[x^g(f)=\epsilon_{\mathbb{G}}(f)+\epsilon_1 g(\sum (-1)^{|f_{(1)}|}S_{\mathbb{G}}(f_{(1)})x(f_{(2)})f_{(3)}), f\in\Bbbk[\mathbb{G}], g\in\mathbb{G}(A),\]
		hence the Hopf superalgebra $\Bbbk[\mathrm{Cent}_{\mathbb{G}}(x)]$ is isomorphic to $\Bbbk [\mathbb{G}]/J$, where the Hopf superideal $J$ is generated by the elements
		\[x(f)-\sum (-1)^{|f_{(1)}|}S_{\mathbb{G}}(f_{(1)})x(f_{(2)})f_{(3)}, f\in\Bbbk[\mathbb{G}].\] 	
		Further, we have
		\[x(f)-\sum (-1)^{|f_{(1)}|}S_{\mathbb{G}}(f_{(1)})x(f_{(2)})f_{(3)}=x(f)-\sum (-1)^{|f_{(1)}|}S_{\mathbb{G}}(f_{(1)}) (x\cdot f_{(2)})\]
		\[-\sum (-1)^{|f_{(1)}|+|f_{(2)}|}S_{\mathbb{G}}(f_{(1)})f_{(2)}x(f_{(3)})=\sum (-1)^{|f_{(1)}|}S_{\mathbb{G}}(f_{(1)}) (x\cdot f_{(2)}),\]
		that is $J\subseteq I$. 
		
		Conversely, let $f=x\cdot h=-\sum x(h_{(1)})h_{(2)}+\sum (-1)^{|h_{(1)}|}h_{(1)}x(h_{(2)})\in \mathrm{Im} (x_{\Bbbk[\mathbb{G}]}), h\in\Bbbk[\mathbb{G}]$. Then, considering all modulo $J$, we have
		\[\sum (-1)^{|h_{(1)}|} h_{(1)}x(h_{(2)})\equiv \sum(-1)^{|h_{(1)}|+|h_{(2)}|}h_{(1)}S_{\mathbb{G}}(h_{(2)})x(h_{(3)})h_{(4)}= \]
		\[\sum\epsilon_{\mathbb{G}}(h_{(1)})x(h_{(2)})h_{(3)}=\sum x(h_{(1)})h_{(2)},\]
		that is $I\subseteq J$. Lemma is proved. 	
	\end{proof}
	\begin{pr}\label{C_G(x) to another G_x}
		There is a natural supergroup morphism $\mathrm{Cent}_{\mathbb{G}}(x)\to \widetilde{\mathbb{G}_x}$, that factors through $\mathbb{G}_x$. Moreover, for any
		$\mathbb{G}$-supermodule $M$, the action of $\mathbb{G}_x$ on $M_x$ factors through the action of $\widetilde{\mathbb{G}_x}$.	
	\end{pr}
	\begin{proof}
		The canonical superalgebra morphism $\ker(x_{\Bbbk[\mathbb{G}]})\to \Bbbk[\mathrm{Cent}_{\mathbb{G}}(x)]$ vanishes on $\mathrm{Im} (x_{\Bbbk[\mathbb{G}]})$, and since the Hopf superalgebra structures on both  $\Bbbk[\mathrm{Cent}_{\mathbb{G}}(x)]$ and $\Bbbk[\widetilde{\mathbb{G}_x}]$ are induced by the Hopf superalgebra structure on $\Bbbk[\mathbb{G}]$, the first statement follows.  
		
		Next, let $\mathbb{H}$ denote the kernel of the above morphism $\mathrm{Cent}_{\mathbb{G}}(x)\to\widetilde{\mathbb{G}_x}$. Then $\Bbbk[\mathrm{Cent}_{\mathbb{G}}(x)/\mathbb{H}]$ is the Hopf supersubalgebra of $\Bbbk[\mathrm{Cent}_{\mathbb{G}}(x)]$, that coincides with the image of $\ker(x_{\Bbbk[\mathbb{G}]})$ therein. Slightly abusing notations, we identify the elements of  $\ker(x_{\Bbbk[\mathbb{G}]})$ with their images in $\Bbbk[\mathrm{Cent}_{\mathbb{G}}(x)]$. Then the defining superideal of $\mathbb{H}$ 
		is generated by $\ker(x_{\Bbbk[\mathbb{G}]})^+=\ker(x_{\Bbbk[\mathbb{G}]})\cap\ker\epsilon_{\mathbb{G}}$ (cf. \cite{zub4}). In particular, an element $y\in\mathrm{Lie}(\mathrm{Cent}_{\mathbb{G}}(x))$ belongs to $\mathrm{Lie}(\mathbb{H})$ if and only if $y(\ker(x_{\Bbbk[\mathbb{G}]})^+)=0$. 
		
		Consider an element $[y, x]\in [\mathfrak{g}, x]$. For any $h\in \Bbbk[\mathbb{G}]$ we have
		\[ [y, x](h) =\sum (-1)^{|h_{(1)}|}y(h_{(1)})x(h_{(2)})-\sum (-1)^{|y|+|y||h_{(1)}|} x(h_{(1)})y(h_{(2)})=\]
		\[y(\sum (-1)^{|h_{(1)}|}h_{(1)}x(h_{(2)})-\sum x(h_{(1)})h_{(2)})=y(x\cdot h).\]
		Thus $[\mathfrak{g}, x]\subseteq \mathrm{Lie}(\mathbb{H})$. Arguing as in Lemma \ref{DS functor for supergroups}, we complete the proof of the second statement. 
		The third statement is obvious.
	\end{proof}
	\begin{cor}\label{when G_x is a subgroup}
		If there is a $\mathbb{G}$-supermodule $M$ such that $M_x$ is a faithful $\mathbb{G}_x$-supermodule, then the supergroup morphism $\mathbb{G}_x\to\widetilde{\mathbb{G}_x}$ is an embedding.	
	\end{cor}
	\begin{rem}\label{trivial remark about I}
		The superideal $I$ is generated by the elements $x\cdot f$, where $f$ runs over a set of generators of the superalgebra $\Bbbk[\mathbb{G}]$. 	
	\end{rem}
	
	\section{Auxiliary computations}
	
	Let $x$ be a square zero odd endomorphism of a superspace $V$. Let $\mathrm{Sym}(V)$ denote the symmetric superalgebra, freely generated by $V$. Then the action of $x$ on $V$ can be naturally extended to an odd superderivation of $\mathrm{Sym}(V)$. 
	
	Arguing as in Lemma \ref{elementary prop-s}, we choose the homogeneous basis 
	\[y_1, \ldots, y_s, y_{s+1}, \ldots, y_{s+t}\] of a free $\Bbbk[x]$-supersubmodule $F$ of $V$, such that
	$V=V_x\oplus F$. Besides, $|y_k|=0$ if $1\leq k\leq s$, otherwise $|y_k|=1$. 
	
	Note that $\mathrm{Sym}(V)\simeq \mathrm{Sym}(V_x)\otimes\mathrm{Sym}(F)$. Moreover, the superderivation $x$ acts as
	\[x\cdot (f\otimes g)=(-1)^{|f|} f\otimes x\cdot g, f\in\mathrm{Sym}(V_x), g\in \mathrm{Sym}(F),  \]
	hence $\mathrm{Sym}(V)_x\simeq \mathrm{Sym}(V_x)\otimes \mathrm{Sym}(F)_x$. In other words, without loss of generality, one can assume that $V_x=0$.	
	Let $U$ denote the superspace $\sum_{1\leq k\leq s+t}\Bbbk y_k$. 
	
	If $char\Bbbk=0$, then $\mathrm{Sym}(F)$, regarded as a complex with respect to differential $x$, is the tensor product of the  \emph{Koszul complex} \[\mathrm{K}(U_1)=\{\mathrm{Sym}(x\cdot U_1)\otimes \Lambda^q(U_1) \}_{0\leq q\leq t}\] (of free $\mathrm{Sym}(x\cdot U_1)$-modules) and the \emph{de Rham complex}
	\[\mathrm{R}(U_0)=\{\mathrm{Sym}(U_0)\otimes \Lambda^q((x\cdot U_0)) \}_{0\leq q\leq s}.\] They are extended to the exact complexes
	\[0\to \mathrm{Sym}(x\cdot U_1)\otimes \Lambda^t(U_1)\to\ldots\to \mathrm{Sym}(x\cdot U_1)\otimes U_1\to\mathrm{Sym}(x\cdot U_1)\to \Bbbk\to 0 \]
	and
	\[0\to\Bbbk\to \mathrm{Sym}(U_0)\to \mathrm{Sym}(U_0)\otimes x\cdot U_0\to\ldots\to \mathrm{Sym}(U_0)\otimes \Lambda^s(x\cdot U_0)\to 0. \]
	In other words, $\mathrm{H}^{\bullet}(\mathrm{K}(U_1))\simeq\Bbbk$ and $\mathrm{H}^{\bullet}(\mathrm{R}(U_0))\simeq\Bbbk$. Applying Corollary \ref{for complexes} and Remark \ref{graded complement}, we obtain
	$\Bbbk= \mathrm{Sym}(F)_x$. More generally, $\mathrm{Sym}(V)_x\simeq \mathrm{Sym}(V_x)$.
	
	If $char\Bbbk =p>0$, then let $\mathrm{Sym}(U_0^p)$ denote the subalgebra of $\mathrm{Sym}(U_0)$, generated by $p$-powers of the elements from $U_0$. Since $x$ vanishes on  $\mathrm{Sym}(U_0^p)$, $\mathrm{Sym}(F)$ is isomorphic (as a complex) to the tensor product of the same Koszul complex $\mathrm{K}(U_1)$ and the \emph{$p$-restricted} de Rham complex
	$\mathrm{R}_{(p)}(U_0)=\{\mathrm{Sym}_{(p)}(U_0)\otimes \Lambda^q(x\cdot U_0) \}_{0\leq q\leq s}$, where $\mathrm{Sym}_{(p)}(U_0)=\sum_{0\leq k_1, \ldots, k_s< p}\Bbbk y_1^{k_1}\ldots y_s^{k_s}$, then tensored by $\mathrm{Sym}(U_0^p)$. 
	
	Moreover, we have a basis of $\mathrm{Sym}_{(p)}(U_0)$ consisting of \emph{divided power monomials} $y_1^{(k_1)}\ldots y_s^{(k_s)}=\frac{y^{k_1}}{k_1!}
	\ldots \frac{y^{k_s}}{k_s!}$, such that \[x\cdot (y_1^{(k_1)}\ldots y_s^{(k_s)})=y_1^{(k_1-1)}\ldots y_s^{(k_s)}\otimes x\cdot y_1+\ldots +y_1^{(k_1)}\ldots y_s^{(k_s-1)}\otimes x\cdot y_s.\]
	\begin{lm}\label{exactness of p-restricted de Rham}
		The cohomology group of $p$-restricted de Rham complex is isomorphic to $\Lambda(\sum_{1\leq i\leq s} \Bbbk y_{i}^{(p-1)}\otimes  x\cdot y_{i})$ as $\mathbb{N}$-graded (super)space.
	\end{lm}
	\begin{proof}
		Use the induction on $s$, Corollary \ref{for complexes} and Remark \ref{graded complement}. 
	\end{proof}
	\begin{rem}\label{subcomplex of p-restricted de Rham}
		Let $\mathrm{R}'_{(p)}(U_0)$ denote the subcomplex of $\mathrm{R}_{(p)}(U_0)$, spanned by the elements $y_{i_1}^{(k_1)}\ldots y_{i_r}^{(k_r)}\otimes  x\cdot y_{j_1}\wedge \ldots\wedge x\cdot y_{j_l}$, such that either some exponent $k_i$ is strictly less than $p-1$, or $\{i_1, \ldots, i_r\}\neq \{j_1, \ldots, j_l\}$. Then
		\[\mathrm{R}_{(p)}(U_0)=\Lambda(\sum_{1\leq i\leq s} \Bbbk y_i^{(p-1)}\otimes x\cdot y_i)\oplus\mathrm{R}'_{(p)}(U_0), \ x\cdot\Lambda(\sum_{1\leq i\leq s} \Bbbk y_i^{(p-1)}\otimes x\cdot y_i)=0, \]
		and $H^{\bullet}(\mathrm{R}'_{(p)}(U_0))=0$.	
	\end{rem}
	\begin{cor}\label{Sym(V)_x in positive char}
		If $char\Bbbk=p>0$, then \[\mathrm{Sym}(V)_x\simeq\mathrm{Sym}(V_x)\otimes \mathrm{Sym}(U_0^p)\otimes\Lambda(\sum_{1\leq i\leq s}\Bbbk y_i^{(p-1)}\otimes x\cdot y_i).\]
	\end{cor}
	\begin{proof}
		By Corollary \ref{for complexes} and Remark \ref{graded complement}, there is \[\mathrm{H}^{\bullet}(\mathrm{K}(U_1)\otimes \mathrm{R}_{(p)}(U_0))\simeq \Bbbk\otimes\Lambda(\sum_{1\leq i\leq s} \Bbbk y_i^{(p-1)}\otimes x\cdot y_i).\] 
	\end{proof}
\begin{rem}\label{Cartier isomorphism}
The statement of this corollary is well known as {\bf Cartier isomorphism} $\mathrm{R}(U_0)\to \mathrm{H}^{\bullet}(\mathrm{R}(U_0))$, defined by
\[y_1^{k_1}\ldots y_s^{k_s}\otimes x\cdot y_{i_1}\wedge\ldots\wedge x\cdot y_{i_t}\mapsto y_1^{pk_1}\ldots y_s^{pk_s}y_{i_1}^{(p-1)}\ldots y_{i_t}^{(p-1)}\otimes x\cdot y_{i_1}\wedge\ldots\wedge x\cdot y_{i_t},\]
\[0\leq k_1, \ldots, k_s, 1\leq i_1<\ldots<i_t\leq s.\]  	
\end{rem}
The following lemma is a folklore.
	\begin{lm}\label{subcoalgebra generates}
	Let $A$ be a Hopf superalgebra and $B$ be a supersubcoalgebra of $A$. Set $\overline{B}=\{\overline{b}=b-\epsilon_A(b)\mid b\in B \}$. Then $\overline{B}$ and $B$ are right $A$-supercomodules with respect to the right conjugation coaction $\tau_{r-con}$.	
\end{lm}
\begin{proof}
	For any $b\in B$ we have 
	\[\tau_{r-con}(b)=\sum (-1)^{|b_{(1)}||b_{(2)}|}b_{(2)}\otimes S_{\mathbb{G}}(b_{(1)})b_{(3)}= \]
	\[\sum (-1)^{|b_{(1)}||b_{(2)}|}\overline{b_{(2)}}\otimes S_{\mathbb{G}}(b_{(1)})b_{(3)}+\epsilon_{\mathbb{G}}(b)\tau_{r-con}(1).\]	
\end{proof}

	Let $\mathbb{G}$ be an algebraic supergroup, such that $\Bbbk[\mathbb{G}]\simeq \mathrm{Sym}(V)_d$, where $V$ is a supersubcoalgebra in $\Bbbk[\mathbb{G}]$ and $d\in \mathrm{Sym}(V)$ such that $\overline{d}\in \mathrm{Sym}(V)\overline{V}$. Let $x\in\mathfrak{g}_1$ with $[x, x]=0$. 
	\begin{lm}\label{V_x is a subcoalgebra}
		$V_x$ is a supersubcoalgebra in $\Bbbk[\widetilde{\mathbb{G}_x}]$, as well as in $\Bbbk[\mathrm{Cent}_{\mathbb{G}}(x)]$.	
	\end{lm}	
	\begin{proof}
		Let $C$ denote $\mathrm{Sym}(V)$. Let $v_1, \ldots, v_k$ be a homogeneous basis of the superspace $V_x$. For any $v\in V$, there is 
		\[\Delta_{\mathbb{G}}(v)\equiv \sum_{1\leq i, j\leq k}\alpha_{ij} v_i\otimes v_j+\sum_{1\leq i\leq k, 1\leq j'\leq t}\beta_{ij'} v_i\otimes y_{j'}+\]	
		\[\sum_{1\leq i\leq k, 1\leq j'\leq t}\gamma_{ij'} y_{j'}\otimes v_i + \sum_{1\leq j', j"\leq t}\delta_{j'j"} y_{j'}\otimes y_{j"}+\]
		\[\sum_{1\leq j', j"\leq t}\kappa_{j' j"}y_{j'}\otimes x\cdot y_{j"}+ \sum_{1\leq j', j"\leq t}\pi_{j' j"}x\cdot y_{j'}\otimes y_{j"} \pmod{\mathrm{Im}(x_{C\otimes C})}.\]
		If $v\in V_x$, then $\Delta_{\mathbb{G}}(v)\in \ker(x_{C\otimes C})$, that implies \[\beta_{ij'}=\gamma_{ij'}=\delta_{j' j"}= 0, \kappa_{j' j"}=(-1)^{|u_{j'}|}\pi_{j' j"},  1\leq i\leq k, 1\leq j', j"\leq t,\]
		that is
		\[\sum_{1\leq j', j"\leq t}\kappa_{j' j"}y_{j'}\otimes x\cdot y_{j"}+ \sum_{1\leq j', j"\leq t}\pi_{j' j"}x\cdot y_{j'}\otimes y_{j"}=x\cdot(\sum_{1\leq j', j"\leq t}\kappa_{j' j"}y_{j'}\otimes y_{j"}).\]
		The second statement follows by Proposition \ref{C_G(x) to another G_x}.
	\end{proof}
If $char\Bbbk=p>0$, then, without loss of generality, one can replace $d$ by $d^p$ and thus assume that $x\cdot d=0$.
Combining with Lemma \ref{subcoalgebra generates}, we obtain
\[\Bbbk[\widetilde{\mathbb{G}_x}]=\Bbbk[\mathbb{G}]_x\simeq (\mathrm{Sym}(V)_x)_d\simeq (\mathrm{Sym}(V_x)\otimes \mathrm{Sym}(U_0^p)\otimes \Lambda(\sum_{1\leq i\leq s} \Bbbk y_i^{(p-1)}\otimes x\cdot y_i))_d,\]
where $d$ is considered modulo $\mathrm{Im}(x_{\mathrm{Sym}(V)})_d$.
	\begin{pr}\label{coproduct}
		If $char\Bbbk=p>0$, then the coproduct in $\Bbbk[\widetilde{\mathbb{G}_x}]$ is determined as follows. Let $y_i\in U_0$ and $\Delta_{\mathbb{G}}(y_i)=\sum (y_i)_{(1)}\otimes (y_i)_{(2)}$, where
		all tensor factors are (up to a scalar muliple) basic elements of $V_x$, $U$ or $x\cdot U$. 
		There is
		\[(1) \ \Delta_{\widetilde{\mathbb{G}_x}}(y_i^{p-1} x\cdot y_i)=\sum_{(y_i)_{(1)}\in U_0, (y_i)_{(2)}\in U_0} (y_i)_{(1)}^{p-1}x\cdot (y_i)_{(1)}\otimes (y_i)_{(2)}^p+\]
		\[\sum_{(y_i)_{(1)}\in U_0, (y_i)_{(2)}\in (V_x)_0} (y_i)_{(1)}^{p-1}x\cdot (y_i)_{(1)}\otimes\underbrace{(\ldots (y_i)_{(2)}\ldots)}_{p \ \mbox{factors}}+\]
		\[\sum_{(y_i)_{(1)}\in U_0, (y_i)_{(2)}\in U_0} (y_i)_{(1)}^p\otimes (y_i)_{(2)}^{p-1} x\cdot (y_i)_{(2)}+ \]
		\[ \sum_{(y_i)_{(1)}\in (V_x)_0, (y_i)_{(2)}\in U_0} \underbrace{(\ldots (y_i)_{(1)}\ldots)}_{p \ \mbox{factors}}\otimes (y_i)_{(2)}^{p-1} x\cdot (y_i)_{(2)}. \]
		Besides, we have
		\[(2) \ \Delta_{\widetilde{\mathbb{G}_x}}(y_i^p)=\sum_{(y_i)_{(1)}, (y_i)_{(2)}\in (V_x)_0\sqcup U_0} (y_i)_{(1)}^p\otimes (y_i)_{(2)}^p .\] 	
	\end{pr}
	\begin{proof}
		Let $L$ denote the subcomplex 
		\[\mathrm{Sym}(V_x)\otimes \mathrm{Sym}(U_0^p)\otimes\mathrm{R}'_{(p)}(U_0)\oplus \mathrm{Sym}(V_x)\otimes\mathrm{K}^+(U_1)\otimes\mathrm{R}(U_0) ,\]
		where 
		\[\mathrm{K}^+(U_1)=\oplus_{1\leq q\leq t} \mathrm{Sym}(x\cdot U_1)\otimes \Lambda^q(U_1)\oplus \mathrm{Sym}(x\cdot U_1)^+\]
		and $\mathrm{Sym}(x\cdot U_1)^+$ consists of polynomials without constant term. It is clear that $C=C_x\oplus L$ and $L_x=0$.
		
		Let $S$ denote the right hand side of the first formula. We have
		\[\Delta_{\mathbb{G}}(y_i^{p-1} x\cdot y_i)=\]
		\[\sum \pm(\underbrace{\ldots (y_i)_{(1)}\ldots }_{(p-1) \ \mbox{factors}})x\cdot (y_i)_{(1)} \otimes \underbrace{(\ldots (y_i)_{(2)}\ldots)}_{p \ \mbox{factors}}+\]
		\[\sum\pm \underbrace{(\ldots (y_i)_{(1)}\ldots)}_{p \ \mbox{factors}} \otimes (\underbrace{\ldots (y_i)_{(2)}\ldots }_{(p-1) \ \mbox{factors}})x\cdot (y_i)_{(2)},\]
		where the factors in the left/right hand side  products are not necessary the same. One sees that a (nonzero) product 
		\[ \pm(\underbrace{\ldots (y_i)_{(1)}\ldots }_{(p-1) \ \mbox{factors}})x\cdot (y_i)_{(1)}\]
		belongs to the subcomplex $L$,  except the case when it is equal to  $(y_i)_{(1)}^{p-1}x\cdot (y_i)_{(1)}$, where $(y_i)_{(1)}\in U_0$.	
		Furthermore, in the product on the right 
		\[\underbrace{(\ldots (y_i)_{(2)}\ldots)}_{p \ \mbox{factors}}\]
		each factor $(y_i)_{(2)}$ belongs to $(V_x)_0\sqcup U_0\sqcup x\cdot U_1$.
		If at least one $(y_i)_{(2)}$ belongs to $x\cdot U_1$, then this product belongs to $L$. Further, if each $(y_i)_{(2)}$ belongs to $U_0$ and not all of them are the same, then again this product belongs to $L$. Finally, if not all $(y_i)_{(2)}$ belong to $(V_x)_0$, then this product belongs to $L$ as well. 
		To sum it up, each term in the first sum coincides with one of the terms from (1) or it belongs to $L\otimes C+C\otimes L$.
		
		Symmetrically, in the second sum each term coincides with one of the terms from (1) or it belongs to $C\otimes L+L\otimes C$. 
		
		Combining all together, we obtain that $\Delta_{\widetilde{\mathbb{G}_x}}(y_i^{p-1} x\cdot y_i)-S$ belongs to the subcomplex $F=L\otimes C+C\otimes L=L\otimes L\oplus L\otimes C_x\oplus C_x\otimes L$. 
		It remains to note that $\Delta_{\widetilde{\mathbb{G}_x}}(y_i^{p-1} x\cdot y_i)-S\in\ker(x_{C\otimes C})$ and $F_x=0$ by Lemma \ref{elementary prop-s}(1) and Remark \ref{subcomplex of p-restricted de Rham}. 
		
		The proof of the second formula is similar and we leave it for the reader. 
	\end{proof}
	By Remark \ref{trivial remark about I}, $\Bbbk[\mathrm{Cent}_{\mathbb{G}}(x)]\simeq (\mathrm{Sym}(V_x))\otimes \mathrm{Sym}(U))_d$, and as above, $d$ is considered modulo the superideal $I=\mathrm{Sym}(V)_d (x\cdot U)$. Set $W=V_x\cup S_{\mathrm{Cent}_{\mathbb{G}}(x)}(V_x)$.	
	\begin{pr}\label{G_x ?}
The superideal $J$ in $\Bbbk[\mathrm{Cent}_{\mathbb{G}}(x)]$, generated by $\overline{W}$, is a Hopf superideal. Moreover, $J$ defines a normal supersubgroup $\mathbb{H}\unlhd \mathrm{Cent}_{\mathbb{G}}(x)$, such that $\mathrm{Lie}(\mathbb{H})=[\mathfrak{g}, x]$.	
	\end{pr}
\begin{proof}
	First two statements follow by 
	\[\Delta_{\mathrm{Cent}_{\mathbb{G}}(x)}(\overline{w})\subseteq \sum \overline{w_{(1)}}\otimes w_{(2)}+1\otimes\overline{w}, w\in W,  S_{\mathrm{Cent}_{\mathbb{G}}(x)}(\overline{W})\subseteq\overline{W}\]
	and by Lemma \ref{subcoalgebra generates}. 

Note that the maximal superideal $\mathfrak{m}=\ker\epsilon_{\mathrm{Cent}_{\mathbb{G}}(x)}$ is generated by the superspace $\overline{V}$. 
By Lemma \ref{subcoalgebra generates}, both $\overline{V}$ and $\mathfrak{m}$ are  the right supersubcomodules with respect to $\tau_{r-cong}$, hence $\mathfrak{g}^*\simeq \mathfrak{m}/\mathfrak{m}^2\simeq\overline{V}$ as a $\mathbb{G}$-supermodule. Arguing as in Lemma \ref{elementary prop-s}, we have a $\Bbbk[x]$-supermodule decomposition $\mathfrak{g}=\overline{V}_x^*\oplus U^*\oplus (x\cdot U)^*$, so that $[\mathfrak{g}, x]=U^*$. On the other hand, $\mathrm{Lie}(\mathrm{Cent}_{\mathbb{G}}(x))=\mathrm{Cent}_{\mathfrak{g}}(x)=\overline{V}_x^*\oplus U^*$ and $\mathrm{Lie}(\mathbb{H})=U^*$.    	
\end{proof}
\begin{cor}\label{G_x in char=0}
If $char\Bbbk =0$, then $\mathbb{G}_x\simeq \mathrm{Cent}_{\mathbb{G}}(x)/\mathbb{H}^0$.	
\end{cor}	
If $char\Bbbk=p>0$, then even if $\mathbb{H}$ is connected, it is not necessary smallest normal supersubgroup of $\mathrm{Cent}_{\mathbb{G}}(x)$ with the Lie superalgebra to be equal to $[\mathfrak{g}, x]$. In the matter of fact, $\mathbb{H}_1=\mathbb{H}\cap(\mathrm{Cent}_{\mathbb{G}}(x))_1\unlhd \mathrm{Cent}_{\mathbb{G}}(x)$ and $\mathrm{Lie}(\mathbb{H}_1)=\mathrm{Lie}(\mathbb{H})$, where $\mathbb{L}_1$ denotes the first \emph{Frobenius kernel} of a supergroup $\mathbb{L}$ (cf. \cite{jan, zub4}).
Since $\mathbb{H}_1$ is an \emph{infinitesimal supergroup of hight 1} (cf. \cite{dg, zub1}),  Theorem \ref{equivalence} and \cite[Corollary 4.3(a), II, \S 7]{dg} imply
$\mathbb{G}_x\simeq \mathrm{Cent}_{\mathbb{G}}(x)/\mathbb{H}_1$.

	\section{Examples}
	
	\subsection{Minimal rank of $x$}
	
	Let $\mathbb{G}=\mathrm{GL}(m|n)$. 
	
	Recall that $\Bbbk[\mathrm{GL}(m|n)]=\Bbbk[t_{kl}\mid 1\leq k, l\leq m+n]_{d}$, where $|t_{kl}|=0$, provided $1\leq k, l\leq m$ or $m+1\leq k, l\leq m+n$, otherwise $|t_{kl}|=1$, and $d=\det((t_{kl})_{1\leq k, l\leq m})\det((t_{kl})_{m+1\leq k, l\leq m+n})$. The Hopf superalgebra structure of $\Bbbk[\mathbb{G}]$ is given by 
	\[\Delta_{\mathbb{G}}(t_{kl})=\sum_{1\leq s\leq m+n} t_{ks}\otimes t_{sl}, \ \epsilon_{\mathbb{G}}(t_{kl})=\delta_{kl}. \]
	Let $\mathbb{T}=\mathbb{T}(m|n)$ be the maximal torus of $\mathbb{G}$, consisting of all diagonal matrices.
	The superbialgebra $\Bbbk[t_{kl}\mid 1\leq k, l\leq m+n]$ is denoted by $C(m|n)$, or just by $C$.
	
	Let $W$ be the standard $\mathbb{G}$-supermodule of (super)dimension $m|n$ with the (homogeneous) basis $e_1, \ldots , e_{m+n}$, i.e. $|e_k|=|k|=0$ provided $1\leq k\leq m$, otherwise $|e_k|=|k|=1$, and 
	\[\tau_W(e_k)=\sum_{1\leq s\leq m+n} e_s\otimes t_{sk}, 1\leq k\leq m+n.\]
	We also denote $\mathbb{G}$ by $\mathrm{GL}(W)$.
	
	The Lie superalgebra $\mathrm{Lie}(\mathbb{G})=\mathfrak{g}$ is identified with $\mathfrak{gl}(m|n)=\mathrm{End}_{\Bbbk}(W)$, where $\mathrm{End}_{\Bbbk}(W)_{a}$ consists of all endomorphisms $y$, such that $y(W_b)\subseteq W_{a+b\pmod 2}, a, b=0, 1$. Note that if the basic elements $e_{ij}$ of $\mathfrak{g}$ are determined as
	$e_{ij}(t_{kl})=\delta_{ik}\delta_{jl}$, then they act on $W$ by the rule $e_{ij}(e_k)=(-1)^{(|i|+|j|)|i|}\delta_{kj} e_i$. In other words, $e_{ij}$ can be identified with
	the matrix $(-1)^{(|i|+|j|)|i|}E_{ij}\in\mathfrak{gl}(m|n)$.
	
	Set $x=e_{ij}$. Then
	\[(\star) \ x\cdot t_{kl}=\delta_{ik}t_{jl}-(-1)^{|t_{ki}||x|}\delta_{jl}t_{ki}, 1\leq k, l\leq m+n.\]
	\begin{rem}\label{another formula}
		The formula $(\star)$ can be reformulated as follows. Let $\bf T$ denote the \emph{generic matrix} $(t_{kl})_{1\leq k, l\leq m+n}$.	The matrix $\bf T$ can be regarded as an element of
		$\mathrm{GL}(m|n)(C(m|n))$. Set $x\cdot {\bf T}=(x\cdot t_{kl})_{1\leq k, l\leq m+n}$. Then $x\cdot {\bf T}=(-1)^{(|i|+|j|)|i|}[x, {\bf T}]$.	
	\end{rem} 
	From now on we assume that $1\leq i\leq m< j\leq m+n$. 
	In the notations of the previous section, $C$ is isomorphic to $\mathrm{Sym}(V)$, where $V=\sum_{1\leq k, l\leq m+n}\Bbbk t_{kl}$. We also have 
	\[U=\sum_{1\leq l\leq m+n}\Bbbk t_{il}+
	\sum_{1\leq k\neq i, j\leq m+n}\Bbbk t_{kj},\]
	\[x\cdot U=\sum_{1\leq l\neq j\leq m+n}\Bbbk t_{jl}+
	\sum_{1\leq k\neq i, j\leq m+n}\Bbbk t_{ki}+\Bbbk (t_{ii}-t_{jj})\]
	and 
	\[V_x= \sum_{1\leq k,l\neq i, j\leq m+n}\Bbbk t_{kl}.\] 
	In other words, $V$, regarded as a $\Bbbk[x]$-supermodule, is equal to $V_x\oplus\Bbbk[x]\cdot U$.

	By Remark \ref{trivial remark about I}, the defining (Hopf) superideal $I$ of $\mathrm{Cent}_{\mathbb{G}}(x)$ is generated by the elements
	$t_{jl}, l\neq j; \ t_{ki}, k\neq i, j; \ t_{ii}-t_{jj}$.  
	It is also clear that $\ker(x_W)=\sum_{t\neq j}\Bbbk e_t$ and $\mathrm{Im}(x_W)=\Bbbk e_i$, hence $W_x\simeq\sum _{t\neq i, j}\Bbbk e_t$.
	
	By Lemma \ref{V_x is a subcoalgebra}, there holds 
	\[\Delta_{\mathrm{Cent}_{\mathbb{G}}(x)}(t_{kl})=\sum_{1\leq s\neq i, j\leq m+n} t_{ks}\otimes t_{sl}, \ 1\leq k, l\neq i, j\leq m+n .\] 
	Let $d'$ denote $\det((t_{kl})_{1\leq k, l\neq i\leq m})\det((t_{kl})_{m+1\leq k, l\neq j\leq m+n})$. Then $d$, regarded as an element of $\Bbbk[\mathrm{Cent}_{\mathbb{G}}(x)]$,
	is equal to $d' t_{ii}t_{jj}=d' t_{ii}^2$. Moreover, $t_{ii}$ is a group like element and $\Bbbk[V_x]_{d'}$ is a Hopf supersubalgebra of $\Bbbk[\mathrm{Cent}_{\mathbb{G}}(x)]$.
	
	Indeed, the superalgebra embedding $\Bbbk[V_x]_{d'}\to \Bbbk[\mathrm{Cent}_{\mathbb{G}}(x)]$ is dual to the induced supergroup morphism $\mathrm{Cent}_{\mathbb{G}}(x)\to \mathrm{GL}(W_x)$, given by \[g=(g_{kl})_{1\leq k, l\leq m+n}\mapsto g=(g_{kl})_{1\leq k, l\neq i, j\leq m+n}.\] By Proposition \ref{G_x ?}, the kernel of this morphism coincides
	with $\mathbb{H}$, which consists of all matrices 
	$(g_{kl})\in\mathrm{Cent}_{\mathbb{G}}(x)$, such that $g_{kl}\neq 0$ if and only if there is
	\begin{enumerate}
		\item $k\neq l$, then $k=i$ or $l=j$;
		\item $k=l\neq i, j$, then $g_{kk}=1$. 
	\end{enumerate}   
It is clear that $\mathbb{H}$ is connected. Therefore, if $char\Bbbk=0$, then $\mathbb{G}_x\simeq \mathrm{GL}(W_x)\simeq\mathrm{GL}(m-1|n-1)$. By Corollary \ref{when G_x is a subgroup}, $\mathbb{G}_x$ is embedded into $\widetilde{\mathbb{G}_x}$ (see \cite[Theorem 1.2(1)]{asher}). 
	
	Note also that $\mathrm{Cent}_{\mathbb{G}}(x)\to\mathrm{GL}(W_x)$ is split. Indeed, set \[\mathbb{L}=\mathrm{Cent}_{\mathbb{G}}(\Bbbk e_i+\Bbbk e_j)\cap\mathrm{Stab}_{\mathbb{G}}(\sum_{1\leq t\neq i, j\leq m+n}\Bbbk e_t).\] Then $\mathbb{L}\leq\mathrm{Cent}_{\mathbb{G}}(x)$ and the induced morphism $\mathbb{L}\to\mathrm{GL}(W_x)$ is obviously an isomorphism. In other words, $\mathrm{Cent}_{\mathbb{G}}(x)\simeq\mathrm{GL}(W_x)\ltimes\mathbb{H}$.
Further, $\mathbb{H}$ is a semi-direct product of $\mathbb{G}_m$ and its unipotent radical $\mathbb{H}_u$. It is easy to see that $\mathrm{Z}(\mathbb{H}_u)(A)=\{I+e_{ij}\otimes a | a\in A_1\}$ for any superalgebra $A$, that is $\mathrm{Z}(\mathbb{H}_u)\simeq\mathbb{G}_a^-$ and $\mathbb{H}_u/\mathrm{Z}(\mathbb{H}_u)\simeq \mathbb{G}_a^{m+n-2}\times(\mathbb{G}_a^-)^{m+n-2}$.
	
	If $char\Bbbk=p>0$ and $\Bbbk$ is perfect, then $\mathbb{G}_x\simeq \mathrm{GL}(W_x)\ltimes \mathbb{H}/\mathbb{H}_1$, where $\mathbb{H}/\mathbb{H}_1\simeq\mathbb{G}_m\ltimes \mathbb{G}_a^{m+n-2}$.  
	\begin{pr}\label{DS for GL}
		Let $char\Bbbk =p>0$. We have $\widetilde{\mathbb{G}_x}\simeq \mathrm{GL}(W_x)\ltimes \mathbb{R}$, where $\mathbb{R}\simeq \mathbb{G}_m\ltimes\mathbb{R}_u$. In turn, $\mathbb{R}_u$ has a central series with the quotients $\mathbb{G}_a^-\times \mathbb{G}_a^{m+n-2}$ and $(\mathbb{G}_a^-)^{m+n-2}$.	
	\end{pr} 
	\begin{proof}
		As it has been already observed, $\Bbbk[\widetilde{\mathbb{G}_x}]=\Bbbk[\mathbb{G}]_x\simeq (C_x)_{d^p}$.
		By Corollary \ref{Sym(V)_x in positive char}, $C_x$ is isomorphic to a polynomial $\mathrm{Sym}(V_x)\otimes \mathrm{Sym}(U_0^p)$-superalgebra, 
		freely generated by $m+n-1$ odd elements. More precisely, there is
		\[U_0=\sum_{1\leq l\leq m}\Bbbk t_{il}+\sum_{m+1\leq k\neq j\leq m+n} \Bbbk t_{kj},\]
		and the above mentioned odd generators are 
		\[t_{il}^{p-1}t_{jl},  t_{kj}^{p-1}t_{ki}, 1\leq l\leq m, m+1\leq k\neq j\leq m+n. \]
		By Lemma \ref{V_x is a subcoalgebra}, there holds 
		\[\Delta_{\widetilde{\mathbb{G}_x}}(t_{kl})=\sum_{1\leq s\neq i, j\leq m+n} t_{ks}\otimes t_{sl}, \ 1\leq k, l\neq i, j\leq m+n .\] 
		
		Since $d'^p$, and $t_{ii}^p =t_{jj}^p$ are group like elements, we have 
		\[\Bbbk[\widetilde{\mathbb{G}_x}]\simeq  \mathrm{Sym}(V_x)_{d'^p}\otimes \mathrm{Sym}(U_0^p)[t_{il}^{p-1}t_{jl},  t_{kj}^{p-1}t_{ki}]_{t_{ii}^{p}},\] where $\mathrm{Sym}(V_x)_{d'^p}\otimes 1\simeq\Bbbk[\mathrm{GL}(W_x)]$ is a Hopf supersubalgebra of $\Bbbk[\widetilde{\mathbb{G}_x }]$. In other words,
		the supergroup morphism $\mathrm{Cent}_{\mathbb{G}}(x)\to\mathrm{GL}(W_x)$ factors through $\widetilde{\mathbb{G}_x}$. Therefore, $\widetilde{\mathbb{G}_x}\simeq \mathrm{GL}(W_x)\ltimes \mathbb{R}$, where $\mathbb{R}=\ker(\widetilde{\mathbb{G}_x}\to\mathrm{G}(W_x))$. 
		
		By Proposition \ref{coproduct}, the coproduct of $\Bbbk[\widetilde{\mathbb{G}_x}]$ is determined on the rest generators as
		\[(1) \ t_{ii}^p\mapsto t_{ii}^p\otimes t_{ii}^p, \ t_{il}^p\mapsto \sum_{1\leq s\leq m}t_{is}^p\otimes t_{sl}^p,  \ t_{kj}^p\mapsto\sum_{m+1\leq s\leq m+n} t_{ks}^p\otimes t_{sj}^p,\]
		\[(2) \ t_{ii}^{p-1}t_{ji}\mapsto t_{ii}^p\otimes t_{ii}^{p-1}t_{ji}+ t_{ii}^{p-1}t_{ji}\otimes t_{ii}^p, \ t_{il}^{p-1}t_{jl}\mapsto t_{ii}^p\otimes t_{il}^{p-1}t_{jl}+\sum_{1\leq s\leq m} t_{is}^{p-1}t_{js}\otimes t_{sl}^p,\]
		\[(3) \  t_{kj}^{p-1}t_{ki}\mapsto t_{kj}^{p-1}t_{ki}\otimes t_{ii}^p+t_{kj}^p\otimes t_{ii}^{p-1}t_{ji}+\sum_{m+1\leq s\neq j\leq m+n} t_{ks}^p\otimes t_{sj}^{p-1}t_{si}, \]
		\[1\leq l\neq i\leq m, m+1\leq k\neq j\leq m+n.\]
		
		We have $\Bbbk[\mathbb{R}]\simeq\Bbbk[\widetilde{\mathbb{G}_x}]/I$, where the Hopf superideal $I$ is generated by the elements $t_{kl}-\delta_{kl}, 1\leq k, l\neq i, j\leq m+n$. 
		The formulas (1-3) imply 
		\[\Delta_{\widetilde{\mathbb{G}_x}}(t_{il}^p)\equiv t_{ii}^p\otimes t_{il}^p+t_{il}^p\otimes 1 \pmod{I\otimes \Bbbk[\widetilde{\mathbb{G}_x}]+\Bbbk[\widetilde{\mathbb{G}_x}]\otimes I}, \]
		\[\Delta_{\widetilde{\mathbb{G}_x}}(t_{kj}^p)\equiv t_{kj}^p\otimes t_{ii}^p+1\otimes t_{kj}^p \pmod{I\otimes \Bbbk[\widetilde{\mathbb{G}_x}]+\Bbbk[\widetilde{\mathbb{G}_x}]\otimes I},\]
		\[\Delta_{\widetilde{\mathbb{G}_x}}(t_{il}^{p-1}t_{jl})\equiv t_{ii}^p\otimes t_{il}^{p-1}t_{jl} +t_{ii}^{p-1}t_{ji}\otimes t_{il}^p+t_{il}^{p-1}t_{jl}\otimes 1\pmod{I\otimes \Bbbk[\widetilde{\mathbb{G}_x}]+\Bbbk[\widetilde{\mathbb{G}_x}]\otimes I}, \]
		\[\Delta_{\widetilde{\mathbb{G}_x}}(t_{kj}^{p-1}t_{ki})\equiv t_{kj}^{p-1}t_{ki}\otimes t_{ii}^p+ t_{kj}^p\otimes t_{ii}^{p-1}t_{ji}+ 1\otimes t_{kj}^{p-1}t_{ki} \pmod{I\otimes \Bbbk[\widetilde{\mathbb{G}_x}]+\Bbbk[\widetilde{\mathbb{G}_x}]\otimes I},\]
		\[\Delta_{\widetilde{\mathbb{G}_x}}(t_{ii}^{p-1}t_{ji})\equiv t_{ii}^p\otimes t_{ii}^{p-1}t_{ji} +t_{ii}^{p-1}t_{ji}\otimes t_{ii}^p\pmod{I\otimes \Bbbk[\widetilde{\mathbb{G}_x}]+\Bbbk[\widetilde{\mathbb{G}_x}]\otimes I}.\]
		Simplifying our notations, one sees that $\Bbbk[\mathbb{R}]\simeq \Bbbk[u_1, \ldots u_{m+n-1}, v_1, \ldots, v_{m+n-1}]_{u_1}$, where
		$u$-elements are even and $v$-elements are odd, and the Hopf superalgebra structure is given by 
		\[\Delta_{\mathbb{R}}(u_1)=u_1\otimes u_1, \ \Delta_{\mathbb{R}}(u_k)=u_1\otimes u_k +u_k\otimes 1, \ \Delta_{\mathbb{R}}(u_s)=u_s\otimes u_1+1\otimes u_s, \]
		\[\Delta_{\mathbb{R}}(v_1)=v_1\otimes u_1+u_1\otimes v_1,\]
		\[\Delta_{\mathbb{R}}(v_k)=u_1\otimes v_k +v_1\otimes u_k+v_k\otimes 1, \ \Delta_{\mathbb{R}}(v_s)=v_s\otimes u_1+u_s\otimes v_1+1\otimes v_s, \]
		\[2\leq k\leq m, m+1\leq s\leq m+n-1.\]
		It is clear that $\Bbbk[u_1]_{u_1}$ is a Hopf supersubalgebra in $\Bbbk[\mathbb{R}]$, i.e. there is a supergroup epimorphism $\mathbb{R}\to \mathbb{G}_m$, which is obviously split.
		Indeed, let $J$ denote the Hopf superideal, generated by all generators except $u_1$. Then the splitting morphism is dual to $\Bbbk[\mathbb{R}]\to \Bbbk[\mathbb{R}]/J\simeq\Bbbk[\mathbb{G}_m]$. Finally, the defining Hopf superideal of $\mathbb{S}=\ker(\mathbb{R}\to\mathbb{G}_m)$ is generated by $u_1-1$.  Therefore, $\Bbbk[\mathbb{S}]$ is isomorphic to  the polynomial superalgebra $\Bbbk[u_t, v_l\mid 2\leq t\leq m+n-1, 1\leq l\leq m+n-1]$, such that all even elements $u_t$ and the odd element $v_1$ are primitive, but
		\[\Delta_{\mathbb{S}}(v_k)=v_k\otimes 1+1\otimes v_k+v_1\otimes u_k, \ \Delta_{\mathbb{S}}(v_s)=v_s\otimes 1+1\otimes v_s+u_s\otimes v_1.\]
		In particular, $\mathbb{S}$ is not abelian and there is a supergroup epimorphism $\mathbb{S}\to \mathbb{G}_a^-\times \mathbb{G}_a^{m+n-2}$, that is dual to the Hopf superalgebra embedding $\Bbbk[u_t, v_1]\to\Bbbk[\mathbb{S}]$.  The kernel of the latter morphism is obviously isomorphic to $(\mathbb{G}_a^-)^{m+n-2}$.
		
		Finally, recall that the conjugation action of any supergroup on itself is uniquely defined by the (super)comodule map $\tau_{r-con}$. Let $J'$ denote the Hopf superideal of $\Bbbk[\mathbb{S}]$, generated by all $u_t$ and $v_1$. Then
		\[\tau_{r-con}(v_k)=1\otimes v_k+v_k\otimes 1-1\otimes v_k-v_1\otimes u_k+u_k\otimes v_1-1\otimes v_1u_k\equiv\]\[v_k\otimes 1\pmod{\Bbbk[\mathbb{S}]\otimes J'+J'\otimes\Bbbk[\mathbb{S}]},\]
		and similarly
		\[\tau_{r-con}(v_s)\equiv v_s\otimes 1 \pmod{\Bbbk[\mathbb{S}]\otimes J'+J'\otimes\Bbbk[\mathbb{S}]}.\]
		Proposition is proved.
	\end{proof}
	\begin{rem}\label{GL is reductive}
		The supergroup $\mathrm{GL}(m|n)$ is reductive, i.e.
		it does not contain nontrivial normal smooth connected unipotent supersubgroups. Actually, if $\mathbb{R}$ is a such supersubgroup, then, since $\mathrm{GL}(m|n)_{ev}\simeq \mathrm{GL}(m)\times\mathrm{GL}(n)$ is reductive, $\mathbb{R}_{ev}$ is trivial. Thus the Harish-Chandra pair of $\mathbb{R}$ has a form $(1, \mathfrak{h})$, where	
		$\mathfrak{h}$ is an odd superideal of $\mathfrak{g}(m|n)$, that implies $\mathfrak{h}=0$.
	\end{rem}		
	\begin{rem}\label{there is no isomorphism}
		Note that in general the supergroups $\mathrm{Cent}_{\mathbb{G}}(x)$ and $\widetilde{\mathbb{G}_x}$ are not isomorphic! Indeed, since $\mathrm{GL}(W_x)$ is reductive, then $\mathbb{H}_u$ should be mapped isomorphically onto $\mathbb{R}_u$. But $\mathrm{Z}(\mathbb{H}_u)$ is one dimensional, hence Proposition \ref{DS for GL} would infer $m+n\leq 3$.
	\end{rem}
\begin{rem}\label{embedding}
If $char\Bbbk=p>0$, then $\mathbb{G}_x$ is embedded into $\widetilde{\mathbb{G}_x}$ as well. In fact, it is easy to show that the kernel of $\mathrm{Cent}_{\mathbb{G}}(x)\to \widetilde{\mathbb{G}_x}$	coincides with $\mathbb{H}_1$.
\end{rem}	
	
	The above computations show that the supergroup $\widetilde{\mathbb{G}_x}$ may be significantly different from $\mathbb{G}_x$. This raises the natural question of how relevant this difference is? Is there a $\mathbb{G}$-supermodule $M$, such that $M_x$ is a faithful $\widetilde{\mathbb{G}_x}$-supermodule? Below we show that the answer is positive at least in the above particular case.  
	
	Consider the induced action of $\mathbb{G}$ on $\mathrm{Sym}(\Pi W)$. In this case $U=\Bbbk e_i +\Bbbk e_j$, and by Corollary \ref{Sym(V)_x in positive char} there is
	$\mathrm{Sym}(\Pi W)_x\simeq \mathrm{Sym}(\Pi W_x)\otimes\mathrm{Sym}(\Bbbk e_j^p)\otimes\Lambda(\Bbbk e_j^{p-1}e_i)$.  The supercomodule structure of $\mathrm{Sym}(\Pi W)_x$ is determined as
	\[e_l\mapsto \sum_{1\leq k\neq i, j\leq m+n} e_k\otimes t_{kl}, \ 1\leq l\neq i, j\leq m+n, \  e_j^p\mapsto e_j^p\otimes t_{ii}^p+\sum_{m+1\leq k\neq j\leq m+n} e_k^p\otimes t_{kj}^p ,\]
	\[e_j^{p-1}e_i\mapsto \sum_{m+1\leq k\neq j\leq m+n} e_k^p\otimes t_{kj}^{p-1}t_{ki}+e_j^{p-1}e_i\otimes t_{ii}^p .\]
	Symmetrically, consider the induced action of $\mathbb{G}$ on $\mathrm{Sym}(W^*)$. Let $\{e^*_k\}_{1\leq k\leq m+n}$ be a basis of $W^*$, that is dual to the standard one. 
	Then $x\cdot e^*_l=-\delta_{il} e_j^*$, hence $\mathrm{Sym}(W^*)_x\simeq \mathrm{Sym}((W_x)^*)\otimes \mathrm{Sym}(\Bbbk (e_i^*)^p)\otimes\Lambda((e_i^*)^{p-1}e_j^*)$. Moreover, the supercomodule structure of $\mathrm{Sym}(W^*)_x$ is determined as
	\[e_l^*\mapsto \sum_{1\leq k\neq i, j\leq m+n} (-1)^{|k|(|k|+|l|)}e^*_k\otimes S_{\widetilde{\mathbb{G}_x}}(t_{lk}), 1\leq l\neq i, j\leq m+n,\]
	\[(e_i^*)^p\mapsto (e_i^*)^p\otimes t_{ii}^{-p}+\sum_{1\leq k\neq i\leq m} (e_k^*)^p\otimes S_{\widetilde{\mathbb{G}_x}}(t_{ik}^p),\]
	\[ (e_i^*)^{p-1}e_j^*\mapsto\sum_{1\leq k\leq m} (e_k^*)^p\otimes S_{\widetilde{\mathbb{G}_x}}(t_{ik}^{p-1}t_{jk})+(e_i^*)^{p-1}e_j^*\otimes t_{ii}^{-p}. \]
	Set $M=\oplus_{0\leq t\leq p}(\mathrm{Sym}^t(W)\oplus \mathrm{Sym}^t(W^*))$. Our formulas infer that the $\widetilde{\mathbb{G}_x}$-supermodule 
	$M_x$ has a coefficient space, which generates the Hopf superalgebra $\Bbbk[\widetilde{\mathbb{G}_x}]$. In particular, $M_x$ is a faithful $\widetilde{\mathbb{G}_x}$-supermodule.
	
	\subsection{Maximal rank of $x$}
	
	Set $m=n$ and $x=\sum_{1\leq i\leq n} e_{i, n+i}$. If $A\in\mathfrak{gl}(m|n)$ is represented as
	\[A=\left(\begin{array}{cc}
		A_{11} & A_{12} \\
		A_{21} & A_{22} 
	\end{array}\right), A_{11}, A_{12}, A_{21}, A_{22}\in\mathfrak{gl}(n), \]
	then 
	\[[x, A]=\left(\begin{array}{cc}
		A_{21} & A_{11}-A_{22} \\
		0 & A_{21} 
	\end{array}\right).\]
	In particular, we have $\mathrm{Cent}_{\mathfrak{g}}(x)=[\mathfrak{g}, x]$ consists of all matrices $A$ with $A_{21}=0$ and $A_{11}=A_{22}$. Similarly, Remark \ref{trivial remark about I}
	and Remark \ref{another formula} infer that $\mathrm{Cent}_{\mathbb{G}}(x)$ consists of the same type of matrices. Moreover, since $\mathrm{Cent}_{\mathbb{G}}(x)$ is connected, the supergroup $\mathbb{G}_x$ is trivial, provided $char\Bbbk=0$. Otherwise, $\mathbb{G}_x\simeq \mathrm{GL}(n)$. It is also obvious that $W_x=0$. 
	
	Further, we have
	\[U=\sum_{1\leq k, l\leq n}\Bbbk t_{kl} +\sum_{1\leq k\leq n< l\leq 2n}\Bbbk t_{kl}, \]
	\[x\cdot U=\sum_{1\leq k, l\leq n}\Bbbk t_{k+n, l} +\sum_{1\leq k\leq n< l\leq 2n}\Bbbk (t_{k, l-n}-t_{k+n, l}), \ V_x=0.\]
	
	Assume again that $char\Bbbk =p>0$.
	By Corollary \ref{Sym(V)_x in positive char}, $C_x$ is isomorphic to $\Bbbk[ t_{kl}^p, t_{kl}^{p-1}t_{k+n, l} | 1\leq k, l\leq n]$. Since for any $1\leq k, l\leq n$ there is
	\[t_{k+n, l+n}^p=t_{kl}^p+(x\cdot t_{k, l+n})^p\equiv t_{kl}^p \pmod{\mathrm{Im}(x_C)},\] the group like element $d^p$ is congruent to $d'^{2p}$, where $d'=\det((t_{kl})_{1\leq k, l\leq n})$. We have $\Bbbk[\widetilde{\mathbb{G}_x}]\simeq (C_x)_{d'^{2p}}$ and 
	arguing as above, we have
	\[(4) \ \Delta_{\widetilde{\mathbb{G}_x}}(t_{kl}^{p-1}t_{k+n, l})=\sum_{1\leq s\leq n} t_{ks}^{p-1}t_{k+n, s}\otimes t_{sl}^p +\sum_{1\leq s\leq n} t_{ks}^p\otimes t_{sl}^{p-1}t_{s+n, l}, 1\leq k\leq n. \]
	One can immediately conclude that the map $t_{kl}\mapsto t_{kl}^p, t_{k, l+n}\mapsto t_{kl}^{p-1}t_{k+n, l}, 1\leq k, l\leq n$, induces the isomorphism $\Bbbk[\mathrm{Cent}_{\mathbb{G}}(x)]\to\Bbbk[\widetilde{\mathbb{G}_x}]$ of Hopf superalgebras.
	\begin{pr}\label{maximal rank}
		If $m=n$ and $x=\sum_{1\leq i\leq n} e_{i, i+n}$, then $\widetilde{\mathbb{G}_x}\simeq\mathrm{Cent}_{\mathbb{G}}(x)\simeq \mathrm{GL}(n)\ltimes (\mathbb{G}_a^-)^{n^2}$.	
	\end{pr}
	
	\subsection{Queer supergroup}
	
	Let $\mathbb{G}$ be a queer supergroup $\mathrm{Q}(n)$. Recall that $\mathrm{Q}(n)\leq\mathrm{GL}(n|n)$ and a matrix $A\in\mathrm{GL}(n|n)$ belongs to $\mathrm{Q}(n)$ if and only if
	$A_{11}=A_{22}, A_{12}=-A_{21}$. 
	Slightly modifying the pevious notations, we have \[\Bbbk[\mathrm{Q}(n)]=\Bbbk[s_{kl}, s'_{kl}| 1\leq k, l\leq n]_{d}, \ \mbox{where} \ d=\det((s_{kl})_{1\leq k, l\leq n}),\] subject to
	\[\Delta_{\mathbb{G}}(s_{kl})=\sum_{1\leq t\leq n} s_{kt}\otimes s_{tl}-\sum_{1\leq t\leq n} s'_{kt}\otimes s'_{tl}, \  \Delta_{\mathbb{G}}(s'_{kl})=\sum_{1\leq t\leq n} s'_{kt}\otimes s_{tl}+\sum_{1\leq t\leq n} s_{kt}\otimes s'_{tl},\]
	\[1\leq k, l\leq n.\]
	Besides, $|s_{kl}|=0, |s'_{kl}|=1, 1\leq k, l\leq n$. Following these notations, we consider the elements of $\mathbb{G}$ as \emph{couples} $((g_{kl})_{1\leq k, l\leq n}|(g'_{kl})_{1\leq k,l\leq n})$ of $n\times n$ matrices with even and odd entries respectively. 
	
	The Lie superalgebra $\mathfrak{q}(n)=\mathrm{Lie}(\mathbb{G})$ has a basis $e_{kl}, e'_{kl}$, such that 
	\[e_{kl}(s_{uv})=\delta_{ku}\delta_{lv}, \ e_{kl}(s'_{uv})=0, \ e'_{kl}(s'_{uv})=\delta_{ku}\delta_{lv}, \ e'_{kl}(s_{uv})=0, 1\leq k, l, u, v\leq n.  \]
	Note that each $e_{kl}$ corresponds to the matrix $E_{kl}+E_{k+n, l+n}$, as well as each $e'_{kl}$ corresponds to the matrix $E_{k+n, l}+E_{k, l+n}$. 
	
	Set $x=e'_{ij}, i\neq j$. Then \[x\cdot s_{kl}=-\delta_{ik}s'_{jl}-\delta_{jl}s'_{ki}, \ x\cdot s'_{kl}=\delta_{ik}s_{jl}-\delta_{jl}s_{ki}, 1\leq k, l\leq n. \]
		As above, let $V$ denote $\sum_{1\leq k, l\leq n}\Bbbk s_{kl} +\sum_{1\leq k, l\leq n}\Bbbk s'_{kl}$. Then
	\[U=\sum_{1\leq l\leq n}\Bbbk s_{il}+\sum_{1\leq l\leq n}\Bbbk s'_{il}+\sum_{1\leq k\neq i, j\leq n}\Bbbk s_{kj}+\sum_{1\leq k\neq i, j\leq n}\Bbbk s'_{kj},\]
	\[x\cdot U=\sum_{1\leq l\neq j\leq n}\Bbbk s'_{jl}+\sum_{1\leq l\neq j\leq n}\Bbbk s_{jl}+\sum_{1\leq k\neq i, j\leq n}\Bbbk s'_{ki}+\sum_{1\leq k\neq i, j\leq n}\Bbbk s_{ki}+\]
	\[\Bbbk(s'_{jj}+s'_{ii})+\Bbbk(s_{jj}-s_{ii}),\]
	\[V_x=\sum_{1\leq k,l\neq i, j\leq n}\Bbbk s_{kl}+\sum_{1\leq k,l\neq i, j\leq n}\Bbbk s'_{kl}.\]
	
	By Remark \ref{trivial remark about I}, $\mathrm{Cent}_{\mathbb{G}}(x)$ consists of all couples $((g_{kl})|(g'_{kl}))$, such that $g_{jl}=g_{ki}=g'_{jl}=g'_{ki}=0$ for any $1\leq l\neq j, k\neq i\leq n$ and $g_{ii}=g_{jj}, g'_{ii}=-g'_{jj}$. There is $W_x\simeq (\sum_{1\leq t\neq j, j+n\leq 2n}\Bbbk e_t)/(\Bbbk e_{i+n}+\Bbbk e_i)$. Applying  Lemma \ref{V_x is a subcoalgebra} or checking directly, we obtain the natural supergroup morphism
	\[\mathrm{Cent}_{\mathbb{G}}(x)\to \mathrm{GL}(W_x)\simeq\mathrm{GL}(n-1|n-1),\]
	given by $(g_{kl}|(g'_{kl}))\mapsto ((g_{kl})_{k, l\neq i, j}| (g'_{kl})_{k, l\neq i, j})$, that induces epimorphism onto $\mathrm{Q}(n-1)\leq\mathrm{GL}(n-1|n-1)$. 
	Proposition \ref{G_x ?} implies that  the supergroup $\mathbb{H}$ therein is the kernel of this morphism. More precisely, a couple $((g_{kl})|(g'_{kl}))\in\mathrm{Cent}_{\mathbb{G}}(x)$ belongs to $\mathbb{H}$ if (and only if) the following conditions hold :
	\begin{enumerate}
		\item if $k\neq l$ and $g_{kl}\neq 0$, then $k=i$ or $l=j$;
		\item if $k=l$, then $g_{kk}=1$ for $k\neq i, j$;
		\item if $k\neq l$ and $g'_{kl}\neq 0$, then $k=i$ or $j=l$;
		\item if $k=l$, then $g'_{kk}=0$ for $k\neq i, j$.
	\end{enumerate}
	Similarly to the case $\mathbb{G}=\mathrm{GL}(m|n)$, one can show that
	the induced morphism $\mathrm{Cent}_{\mathbb{G}}(x)\to\mathrm{Q}(n-1)$ is split, that is $\mathrm{Cent}_{\mathbb{G}}(x)\simeq\mathrm{Q}(n-1)\ltimes\mathbb{H}$. 
	
	Further, the supergroup
	$\mathbb{H}$ is isomorphic to a semi-direct product $\mathrm{Q}(1)\ltimes\mathbb{H}_u$, where $\mathrm{Q}(1)$ is identified with the supersubgroup of $\mathbb{H}$ consisting of all couples
	$((g_{kl})|(g'_{kl}))$, such that $g_{il}=g_{kj}=g'_{il}=g'_{kj}=0$ for any $1\leq l\neq i, k\neq j\leq n$.
	
	Further, one can show that $\mathrm{Z}(\mathbb{H}_u)$ consists of all couples $((g_{kl})|(g'_{kl}))\in\mathbb{H}_u$ with the only "non-diagonal" entries $g_{ij}$ and $g'_{ij}$ to be nonzero. Thus $\mathrm{Z}(\mathbb{H}_u)\simeq\mathbb{G}_a\times\mathbb{G}_a^-$ and $\mathbb{H}_u/\mathrm{Z}(\mathbb{H}_u)\simeq \mathbb{G}_a^{2n-4}\times (\mathbb{G}_a^-)^{2n-4}$.
	
	Finally, $\mathbb{G}_x\simeq\mathrm{Q}(n-1)$ if $char\Bbbk=0$, otherwise $\mathbb{G}_x\simeq\mathrm{Q}(n-1)\ltimes \mathbb{H}/\mathbb{H}_1$.
	
	Let $char\Bbbk=p>0$. Again, let $d'$ denote $\det((s_{kl})_{1\leq k, l\neq i, j\leq n})$. Then $d^p\equiv d'^p s_{ii}^{2p}\pmod{\mathrm{Im}(x_C)}$, where $C=\Bbbk[s_{kl}, s'_{kl}| 1\leq k, l\leq n]$. 
	Besides, $d'^p$ and $s_{ii}^p$ are group like elements in the superbialgebra $C_x$. 
	\begin{lm}\label{comultiplication in the queer case}
		The Hopf superalgebra structure of $\Bbbk[\widetilde{\mathbb{G}_x}]$ is given by		
		\[(1) \ \Delta_{\widetilde{\mathbb{G}_x}}(s_{kl})=\sum_{1\leq t\neq i, j\leq n} s_{kt}\otimes s_{tl}-\sum_{1\leq t\neq i, j\leq n} s'_{kt}\otimes s'_{tl},\]
		\[\Delta_{\widetilde{\mathbb{G}_x}}(s'_{kl})=\sum_{1\leq t\neq i, j\leq n} s'_{kt}\otimes s_{tl}+\sum_{1\leq t\neq i, j\leq n} s_{kt}\otimes s'_{tl}, \ 1\leq k, l\neq i, j\leq n;\]
		\[(2) \ \Delta_{\widetilde{\mathbb{G}_x}}(s_{ii}^p)=s_{ii}^p\otimes s_{ii}^p, \ \Delta_{\widetilde{\mathbb{G}_x}}(s_{ij}^p)=s_{ii}^p\otimes s_{ij}^p+s_{ij}^p\otimes s_{ii}^p+\sum_{1\leq t\neq i, j\leq n}s_{it}^p\otimes s_{tj}^p,\]
		\[\Delta_{\widetilde{\mathbb{G}_x}}(s_{il}^p)=\sum_{1\leq t\neq j\leq n}s_{it}^p\otimes s_{tl}^p, \ 1\leq l\neq i, j\leq n, \]
		\[\ \Delta_{\widetilde{\mathbb{G}_x}}(s_{kj}^p)=s_{kj}^p\otimes s_{ii}^p+\sum_{1\leq t\neq i, j\leq n} s_{kt}^p\otimes s_{tj}^p, 1\leq k\neq i, j\leq n;\]
		\[(3) \ \Delta_{\widetilde{\mathbb{G}_x}}(s_{il}^{p-1}s'_{jl})=\sum_{1\leq t\neq j\leq n} s_{it}^{p-1}s'_{jt}\otimes s_{tl}^p +s_{ii}^p\otimes s_{il}^{p-1}s'_{jl},\]\[ 1\leq l\neq i, j\leq n,\]
		\[\Delta_{\widetilde{\mathbb{G}_x}}(s_{ii}^{p-1}s'_{ji})=s_{ii}^{p-1}s'_{ji}\otimes s_{ii}^p +s_{ii}^p\otimes s_{ii}^{p-1}s'_{ji};\]
		\[(4) \ \Delta_{\widetilde{\mathbb{G}_x}}(s_{kj}^{p-1}s'_{ki})= s_{kj}^{p-1}s'_{ki}\otimes s_{ii}^p+s_{kj}^p\otimes s_{ii}^{p-1}s'_{ji}+\sum_{1\leq t\neq i, j\leq n}
		s_{kt}^p\otimes s_{tj}^{p-1}s'_{ti}, 1\leq k\neq i, j\leq n;	\]
		\[(5) \ \Delta_{\widetilde{\mathbb{G}_x}}(s_{ij}^{p-1}(s'_{jj}+s'_{ii}))=\sum_{1\leq t\neq j\leq n} s_{it}^{p-1}s'_{jt}\otimes s_{tj}^p+s_{ij}^{p-1}(s'_{jj}+s'_{ii})\otimes s_{ii}^p+s_{ii}^p\otimes s_{ij}^{p-1}(s'_{jj}+s'_{ii}).\]
	\end{lm}
	\begin{proof}
Use Lemma \ref{V_x is a subcoalgebra} and Proposition \ref{coproduct}. 		
	\end{proof}
In particular, this lemma implies that $\Bbbk[\widetilde{\mathbb{G}_x}]$ is the tensor
product of the Hopf supersubalgebra \[\Bbbk[s_{kl}, s'_{kl}|1\leq k, l\neq i, j\leq n]_{d'^p}\] and the supersubalgebra \[(\mathrm{Sym}(U_0^p)[s_{il}^{p-1}s'_{jl}, s_{kj}^{p-1}s'_{ki}, s_{ij}^{p-1}(s'_{jj}+s'_{ii}), 1\leq l\neq j\leq n, 1\leq k\neq i, j\leq n])_{s_{ii}^p}.\]

	Arguing as above, we have the supergroup epimorphism $\widetilde{\mathbb{G}_x}\to\mathrm{Q}(n-1)$, that is obviously split. Let $\mathbb{R}$ denote the kernel of it. 
	Then $\Bbbk[\mathbb{R}]\simeq \Bbbk[\widetilde{\mathbb{G}_x}]/I$, where the Hopf superideal $I$ is generated by the elements $s_{kl}-\delta_{kl}, s'_{kl}, 1\leq k,l\neq i, j\leq n$. 
	In other words, $\Bbbk[\mathbb{R}]$ is generated by the elements 
	\[s_{ii}^{-p}, s_{il}^p, s_{kj}^p, s_{it}^{p-1}s'_{jt}, s_{vj}^{p-1}s'_{vi}, s_{ij}^{p-1}(s'_{jj}+s'_{ii}),\]
	\[1\leq l\leq n, 1\leq k\neq i, j\leq n, 1\leq t\neq j\leq n, 1\leq v\neq i, j\leq n,\]
	subject to
	\[\Delta_{\mathbb{R}}(s_{ii}^p)=s_{ii}^p\otimes s_{ii}^p, \ \Delta_{\mathbb{R}}(s_{ij}^p)=s_{ii}^p\otimes s_{ij}^p+s_{ij}^p\otimes s_{ii}^p+\sum_{1\leq t\neq i, j\leq n}s_{it}^p\otimes s_{tj}^p, \]
	\[\Delta_{\mathbb{R}}(s_{il}^p)=s_{il}^p\otimes 1+s_{ii}^p\otimes s_{il}^p, 1\leq l\neq i, j\leq n; \]	
	\[ \Delta_{\mathbb{R}}(s_{kj}^p)= 1\otimes s_{kj}^p+s_{kj}^p\otimes s_{ii}^p, 1\leq k\neq i, j\leq n;\]
	\[\Delta_{\mathbb{R}}(s_{it}^{p-1}s'_{jt})=s_{it}^{p-1}s'_{jt}\otimes 1+ s_{ii}^p\otimes s_{it}^{p-1}s'_{jt} +s_{ii}^{p-1}s'_{ji}\otimes s_{it}^p,\]\[\Delta_{\mathbb{R}}(s_{ii}^{p-1}s'_{ji})=s_{ii}^{p-1}s'_{ji}\otimes s_{ii}^p +s_{ii}^p\otimes s_{ii}^{p-1}s'_{ji}, 1\leq t\neq i, j\leq n;\]
	\[ \Delta_{\mathbb{R}}(s_{vj}^{p-1}s'_{vi})= s_{vj}^{p-1}s'_{vi}\otimes s_{ii}^p+s_{vj}^p\otimes s_{ii}^{p-1}s'_{ji}+
	1\otimes s_{vj}^{p-1}s'_{vi}, 1\leq v\neq i, j\leq n;\]
	\[ \Delta_{\mathbb{R}}(s_{ij}^{p-1}(s'_{jj}+s'_{ii}))=\sum_{1\leq t\neq j\leq n} s_{it}^{p-1}s'_{jt}\otimes s_{tj}^p+s_{ij}^{p-1}(s'_{jj}+s'_{ii})\otimes s_{ii}^p+s_{ii}^p\otimes s_{ij}^{p-1}(s'_{jj}+s'_{ii}).\]
	As above, there is a split epimorphism $\mathbb{R}\to \mathbb{G}_m$. Let $\mathbb{M}$ denote its kernel. Then $\Bbbk[\mathbb{M}]\simeq \Bbbk[\mathbb{R}]/J$, where the Hopf superideal $J$ is generated by the element $s_{ii}^p-1$. Let $u_1, \ldots , u_{n-1}$ and $v_1, \ldots, v_{n-2}$ denote the elements $s_{ij}^p, s_{il}^p$ and  $s_{kj}^p$ accordingly, where $ 1\leq k, l\neq i, j\leq n$. Similarly, let $u'_1, \ldots, u'_{n-1}$ and $v'_1, \ldots, v'_{n-2}$ denote $s_{ii}^{p-1}s'_{ji}, s_{it}^{p-1}s'_{jt}$ and $s_{vj}^{p-1}s'_{vi}$ accordingly, where
	$1\leq t, v\neq i, j\leq n$. Finally, let $z$ denote $s_{ij}^{p-1}(s'_{jj}+s'_{ii})$. We have \[\Bbbk[\mathbb{M}]\simeq \Bbbk[u_k, v_l, u'_k, v'_l, z\mid 1\leq k\leq n-1, 1\leq l\leq n-2]\] subject to
	\[\Delta_{\mathbb{M}}(u_1)=u_1\otimes 1+1\otimes u_1+\sum_{2\leq k\leq n-1} u_k\otimes v_{k-1}, \]
	\[\Delta_{\mathbb{M}}(u_k)=u_k\otimes 1+1\otimes u_k, \ \Delta_{\mathbb{M}}(v_l)=v_l\otimes 1+1\otimes v_l,\] 
	\[\Delta_{\mathbb{M}}(u'_1)=u'_1\otimes 1+1\otimes u'_1, \ \Delta_{\mathbb{M}}(u'_k)=u'_k\otimes 1+1\otimes u'_k+u'_1\otimes u_k,\]
	\[\Delta_{\mathbb{M}}(v'_l)=v'_l\otimes 1+1\otimes v'_l+v_l\otimes u'_1, \ \Delta_{\mathbb{M}}(z)=z\otimes 1+1\otimes z+u'_1\otimes u_1+\sum_{2\leq k\leq n-1}u'_k\otimes v_{k-1},\]
	\[2\leq k\leq n-1, 1\leq l\leq n-2.\]
	Since $\Bbbk[u_k, v_l, u'_1 \mid 2\leq k\leq n-1, 1\leq l\leq n-2]$ is a Hopf supersubalgebra of $\Bbbk[\mathbb{M}]$, there is an epimorphism $\mathbb{M}\to \mathbb{G}_a^{2n-4}\times \mathbb{G}_a^-$ and its kernel is isomorphic to $\mathbb{G}_a\times (\mathbb{G}_a^-)^{2n-3}$. In particular, $\mathbb{M}=\mathbb{R}_u$ is the unipotent radical of $\mathbb{R}$.
	\begin{pr}\label{DS for Q}
		We have $\widetilde{\mathbb{G}_x}\simeq \mathrm{Q}(n-1)\ltimes \mathbb{R}$, where $\mathbb{R}\simeq\mathbb{G}_m\ltimes \mathbb{R}_u$. In turn, $\mathbb{R}_u$ has a central series with the quotients 	$\mathbb{G}_a^{2n-4}\times \mathbb{G}_a^-$ and $\mathbb{G}_a\times (\mathbb{G}_a^-)^{2n-3}$.
	\end{pr}
	
	\subsection{Concluding remarks}
	
	All of the above examples have one common property. Namely, in the positive characteristic $\mathrm{Cent}_{\mathbb{G}}(x)$ and $\widetilde{\mathbb{G}_x}$ have subnormal series whose factors are identical up to a permutation.
	On the other hand, if $char \Bbbk=0$, it does not take place. Indeed, set $\mathbb{G}=\mathrm{GL}(m|n)$ and $x=e_{ij}$. The structure of $\mathrm{Cent}_{\mathbb{G}}(x)$ does not depend on $char \Bbbk$. But \cite[Theorem 1.2(1)]{asher} implies $\mathbb{R}\simeq \mathbb{G}_a^-$!
	
	Nevertheless, it turns out that in any characteristic $\mathrm{Cent}_{\mathbb{G}}(x)$ and $\widetilde{\mathbb{G}_x}$ are semi-direct products of the same reductive supergroup and some normal solvable supergroups. We think that the following hypothesis is quite plausible.
	\begin{hyp}
		If $\mathbb{G}$ is a {\bf quasi-reductive} supergroup in the sense of \cite{taiki, vserg}, then for any $x\in X$ the supergroup $\widetilde{\mathbb{G}_x}$ is a semi-direct product of a quasi-reductive supergroup and some normal smooth connected solvable supergroup (i.e., the solvable radical of $\mathbb{G}$). 	
	\end{hyp}

	\section{Injective $\mathbb{G}$-supermodules and DS-functor}
	
	In the paper \cite{DSfunctor} the main attention is paid to the (finite dimensional) projective modules over Lie superalgebras. However, in the category of supermodules over an algebraic supergroup (as in the category of modules over an algebraic group), the presence of projective objects imposes very strong restrictions on the supergroup (respectively, on the group).
	\begin{lm}\label{existence of projective}
		Let $\mathbb{G}$ be an algebraic supegroup. If there is a nonzero projective $\mathbb{G}$-supermodule, then any injective $\mathbb{G}$-supermodule is projective too.	
	\end{lm}	
	\begin{proof}
		Just superize the arguments from \cite[Lemma 1]{don1}.	
	\end{proof}	
	\begin{pr}\label{equivalent conditions to existence of projectives}
		The following conditions are equivalent : 
		\begin{enumerate}
			\item There is a nozero projective $\mathbb{G}$-supermodule;
			\item $\mathbb{G}$ has an integral;
			\item $\mathbb{G}_{ev}$ has an integral;
			\item Any indecomposable injective $\mathbb{G}$-supermodule is finite dimensional. 
		\end{enumerate}	
	\end{pr}
	\begin{proof}
		Combine Lemma \ref{existence of projective} with \cite[Lemma 7.3 and Proposition 7.5]{mas1}	
	\end{proof}
	If $char\Bbbk=0$, then \cite[Proposition 1]{don1} implies that the first condition in Proposition \ref{equivalent conditions to existence of projectives} is equivalent to $\mathbb{G}_{ev}$ is \emph{linearly reductive}. But contrary to \cite[Theorem (iii)]{don1}, it is not equivalent to the linear reductivity of $\mathbb{G}$ (see the final remarks in \cite[Section 7.3]{mas1}). Further, let $char\Bbbk>0$. Without loss of generality, one can assume that $\Bbbk$ is algebraically closed. Then \cite[Theorem (ii) and Corollary in Section 2]{don1} implies that $G=\mathbb{G}_{ev}$ has an integral if and only if $(G_{red})^0$ is a torus.    
	
Thus, in the category of supermodules over an algebraic supergroup, the concept of injective supermodule is more essential than the concept of a projective one. 	
	\begin{lm}\label{an extension}
		Let $\mathbb{G}$ be an algebraic supergroup and $\mathbb{H}$ be its normal supersubgroup. If $\mathbb{G}$-supermodule $M$ is injective, then $M|_{\mathbb{H}}$ is an injective $\mathbb{H}$-supermodule and $M^{\mathbb{H}}$ is an injective $\mathbb{G}/\mathbb{H}$-supermodule. 	
	\end{lm}
	\begin{proof}
		By \cite[Theorem 5.2(i, iv) and Theorem 6.2]{zub4} the $\mathbb{H}$-supermodule $M|_{\mathbb{H}}$ is injective, provided $\mathbb{G}$-supermodule $M$ is. Moreover, if
		$M|_{\mathbb{H}}$ is injective, then the standard spectral sequence \cite[Proposition 3.1 (2)]{scalazub} 
		\[\mathrm{E}_2^{n, m}=\mathrm{Ext}^n_{_{\mathbb{G}/\mathbb{H}}\mathsf{SMod}}(N, \mathrm{H}^m(\mathbb{H}, M))\Rightarrow\mathrm{Ext}^{n+m}_{_{\mathbb{G}}\mathsf{SMod}}(N, M) \]
		degenerates for any $\mathbb{G}/\mathbb{H}$-supermodule $N$. In other words, we have
		\[\mathrm{Ext}^n_{_{\mathbb{G}/\mathbb{H}}\mathsf{SMod}}(N, M^{\mathbb{H}})\simeq \mathrm{Ext}^{n}_{_{\mathbb{G}}\mathsf{SMod}}(N, M)\]
		for any $n\geq 0$.	Thus, if $M$ is injective as a $\mathbb{G}$-supermodule, then $M^{\mathbb{H}}$ is injective as a $\mathbb{G}/\mathbb{H}$-supermodule.
	\end{proof}
The Duflo-Serganova functor detects the property of a $\mathbb{G}$-supermodule to be injective. 
	\begin{pr}\label{injectives and DS-functor}
		For any injective $\mathbb{G}$-supermodule $M$ there is $X_M=0$. In particular, if $\mathbb{G}$ is regarded as a right $\mathbb{G}$-superscheme with respect to the right multiplication action, then $X_{\Bbbk[\mathbb{G}]}=0$. 
	\end{pr}	
	\begin{proof}
		Let $\mathfrak{g}=\mathrm{Lie}(\mathbb{G})$ and $x\in\mathfrak{g}_1$, such that $[x, x]=0$. 
		Let $\mathbb{U}$ denote the purely odd unipotent supersubgroup of $\mathbb{G}$, that corresponds to the Harish-Chandra subpair $(1, \Bbbk x)$. Since $\mathbb{U}$ is finite, the sheaf quotient
		$\mathbb{G}/\mathbb{U}$ is affine (cf. \cite[Theorem 0.1]{zub1}), hence $\mathbb{U}$ is an (faithfully) exact supersubgroup of $\mathbb{G}$ and $M|_{\mathbb{U}}$ is an injective $\mathbb{U}$-supermodule (use again \cite[Theorem 5.2(i, iv)]{zub4}). Thus $M$ is a direct summand of a direct sum $\Bbbk[\mathbb{U}]^{\oplus I}\oplus\Pi\Bbbk[\mathbb{U}]^{\oplus J}$ for some
		(possibly infinite) sets of indices $I$ and $J$. It remains to note that $\Bbbk[\mathbb{U}]$ is isomorphic to a free $\mathrm{Dist}(\mathbb{U})$-supermodule of rank one, hence 
		$\Bbbk[\mathbb{U}]_x=0$.
	\end{proof}

	\begin{pr}\label{if G is purely odd}
		Let $\mathbb{G}$ be a purely odd (unipotent) supergroup and $M$ be a $\mathbb{G}$-supermodule. Then $M$ is injective if and only if $M=\varinjlim_{i\in I} M_i$, where each supersubmodule $M_i$ is finite dimensional and  $X_{M_i}=0$.  
	\end{pr} 
	\begin{proof}
		The part "if". As it has been alreay observed, any injective $\mathbb{G}$-supermodule is a free $\mathrm{Dist}(\mathbb{G})$-supermodule, hence it is the direct limit of finitely generated free
		$\mathrm{Dist}(\mathbb{G})$-supersubmodules, each of which satisfies the above conditions.
		
		The part "only if". Note that any direct limit of injective $\mathrm{Dist}(\mathbb{G})$-supermodules is again injective (use the \emph{bozonization} arguments as in \cite[Lemma 4.4]{zub3} and \cite[Chap.1, Ex.8]{cart-el}). Therefore, all we need is to prove our proposition for $M$ to be finite dimensional. By Remark \ref{extension invariant}, we also assume that $\Bbbk$ is algebraically closed.
		
		Consider a decomposition $\mathbb{G}=\mathbb{G}_a^-\times(\mathbb{G}_a^-)^{r-1}$ and denote the second factor by $\mathbb{N}$. Choose a basis $z_1, \ldots, z_r$ of $\mathfrak{g}$, such that $\mathrm{Lie}(\mathbb{N})=\sum_{2\leq i\leq r}\Bbbk z_i$. Since $M_{z_1}=0$, we have the exact sequence
		\[0\to z_1 M\to M\to z_1 M\to 0,\]
		where the rightmost map is $m\mapsto z_1 m, m\in M$. The supermodule $z_1 M$ has a natural structure of $\mathrm{Dist}(\mathbb{N})$-supermodule. If we will prove that for any $z\in\mathrm{Lie}(\mathbb{N})$ there is $(z_1 M)_z=0$, then by the induction on $r$ the $\mathrm{Dist}(\mathbb{N})$-supermodule $z_1 M$ is free. The latter easily infers that $M$ is a free $\mathrm{Dist}(\mathbb{G})$-supermodule, hence injective. 
		
		Now, without loss of generality, one can assume that $z=z_2$ and $r=2$. 
		As in Lemma \ref{free basis}, $M=F\oplus G$, where $F$ is free and $(z_1 z_2)G=0$. As a direct summand of $M$, $G$ satisfies the condition $X_{G}=0$. 
		As a $\Lambda(\Bbbk z_1)$-supermodule, $G$ is free. Let $g_1, \ldots , g_s$ be its free generators. The condition $(z_1 z_2)G=0$ implies $zG=z_1G$ for any $z\in\mathfrak{g}$.
		In particular, we have 
		\[z_2 g_i=\sum_{1\leq j\leq s}\alpha_{ij} z_1 g_j, 1\leq i\leq s.\]
		Let $\lambda$ be an eigenvalue of $A=(\alpha_{ji})_{1\leq i, j\leq s}$ and $(\alpha_1, \ldots , \alpha_s)$ be a coordinate vector of the corresponding eigenvector. Set $g=\sum_{1\leq i\leq s} \alpha_i g_i$. Then $(\lambda z_1-z_2)g=0$ but $g$ does not belong to $zG=z_1G$, a contradiction! 
	\end{proof}	
	\begin{rem}\label{not always injective}
		The only condition $X_M=0$ does not imply that $M$ is injective. In fact, let $\mathbb{G}=(\mathbb{G}_a^-)^2$, so that $\mathfrak{g}=\Bbbk z_1+\Bbbk z_2$. Let $M$ has a basis $v_i, w_i, 1\leq i\leq\infty$, such that $|v_i|=0, |w_i|=1, z_1 v_i=w_i, z_2 w_i=0, z_2v_i=w_i+w_{i+1}$ for any index $i$. Then $M$ is a $\mathrm{Dist}(\mathbb{G})$-supermodule, that satisfies 
		$(z_1 z_2)M=0$. In particular, $M$ does not contain any free supersubmodules. On the other hand, for any $z\in\mathfrak{g}$ there obviously holds $M_z=0$.
	\end{rem}
	Assume that the Lie superalgebra $\mathfrak{g}$ of $\mathbb{G}$ satisfies $[\mathfrak{g}_1, \mathfrak{g}_1]=0$.
	This condition is equivalent to the fact that the natural embedding $\mathbb{G}_{ev}\to\mathbb{G}$ has a left inverse $\pi : \mathbb{G}\to\mathbb{G}_{ev}$, or, to the fact that  $\mathbb{G}\simeq\mathbb{G}_{ev}\ltimes \mathbb{G}_{odd}$, where $\mathbb{G}_{odd}\simeq (\mathbb{G}_a^-)^{\dim\mathfrak{g}_1}$ is a purely odd normal unipotent supersubgroup of $\mathbb{G}$ (see \cite[Proposition 6.1]{zub-bov}). We call $\mathbb{G}$ to be \emph{split} (in \cite{zub-bov} it is called \emph{graded}). Observe that if $M$ is a $\mathbb{G}$-supermodule, then $M^{\mathbb{G}_{odd}}$ is naturally a supersubmodule of $M|_{\mathbb{G}_{ev}}$. More precisely, let $A$ denote the Hopf (super)subalgebra
	$\Bbbk[\mathbb{G}]^{\mathbb{G}_{odd}}$, that is isomorphic to $\Bbbk[\mathbb{G}_{ev}]$. Then Remark \ref{a folklore fact} infers $\tau_M(M^{\mathbb{G}_{odd}})\subseteq M^{\mathbb{G}_{odd}}\otimes A$ and the natural superalgebra morphism $A\to \overline{\Bbbk[\mathbb{G}]}\simeq\Bbbk[\mathbb{G}_{ev}]$ is an isomorphism. 
	
	The following lemma is a supergroup analog of \cite[Lemma 10.3]{DSfunctor}.
	\begin{lm}\label{converse of the above prop}
		Let $\mathbb{G}$ be a split supergroup. Then a $\mathbb{G}$-supermodule $M$ is injective if and only if $M$ is injective as a $\mathbb{G}_{odd}$-supermodule, and $M^{\mathbb{G}_{odd}}$ is injective as a $\mathbb{G}_{ev}$-supermodule. Moreover, in the latter case $M$ is isomorphic to  $\mathrm{ind}^{\mathbb{G}}_{\mathbb{G}_{ev}} M^{\mathbb{G}_{odd}}$.
	\end{lm}	
	\begin{proof}
		By Lemma \ref{an extension}, the only "if" part needs proof. Let $V$ denote $M^{\mathbb{G}_{odd}}$. By Frobenius reciprocity (cf. \cite{jan, zub5}), the induced supermodule $\mathrm{ind}^{\mathbb{G}}_{\mathbb{G}_{ev}} V$ is injective. It remains to show that $M\simeq \mathrm{ind}^{\mathbb{G}}_{\mathbb{G}_{ev}} V$.
		
		Since $V$ is an injective $\mathbb{G}_{ev}$-supermodule, there is a $\mathbb{G}_{ev}$-supermodule morphism $\pi : M\to V$, that is a left inverse of the natural embedding $V\to M$. 
		By Frobenius reciprocity, it induces the  $\mathbb{G}$-supermodule morphism $\widehat{\pi} : M\to \mathrm{ind}^{\mathbb{G}}_{\mathbb{G}_{ev}} V$ as
		\[m\mapsto \sum \pi(m_{(1)})\otimes m_{(2)}, m\in M, \tau_M(m)=\sum m_{(1)}\otimes m_{(2)}.\]
		We show that $\widehat{\pi}$ is the required isomorphism. As it has been already noted, both $M|_{\mathbb{G}_{odd}}$ and $(\mathrm{ind}^{\mathbb{G}}_{\mathbb{G}_{ev}} V)|_{\mathbb{G}_{odd}}$ are injective $\mathbb{G}_{odd}$-supermodules. Thus all we need is to prove that $\widehat{\pi}$ induces the isomorphism of their socles 
		$V\to (\mathrm{ind}^{\mathbb{G}}_{\mathbb{G}_{ev}} V)^{\mathbb{G}_{odd}}$. 
		Recall that $\tau_M(V)\subseteq V\otimes A$. Since $\pi|_V=\mathrm{id}_V$, $\widehat{\pi}|_{V}=\tau_M|_V$ is injective. 
		
		By \cite[Lemma 8.2]{zub2}, $\mathrm{ind}^{\mathbb{G}}_{\mathbb{G}_{ev}} V$ is isomorphic to $V\otimes\Bbbk[\mathbb{G}_{odd}]$ (as a superspace). Moreover, \cite[Remark 8.3]{zub2} infers
		that $(\mathrm{ind}^{\mathbb{G}}_{\mathbb{G}_{ev}} V)|_{\mathbb{G}_{odd}}$ is isomorphic to $V_{triv}\otimes\Bbbk[\mathbb{G}_{odd}]$ as a right (super)comodule. Thus
		$(\mathrm{ind}^{\mathbb{G}}_{\mathbb{G}_{ev}} V)^{\mathbb{G}_{odd}}$ can be identified with $V\otimes 1$. Even more precisely, the isomorphism from \cite[Lemma 8.2]{zub2} takes
		each $v\otimes 1\in V\otimes 1$ to $\tau_M(v)$, hence $\widehat{\pi}$ is surjective.  
	\end{proof}
	\begin{rem}\label{above lemma in char=0}
	If $\mathbb{G}_{ev}$ is linearly reductive, then the second codnition of Lemma \ref{converse of the above prop} is superfluous. 	
	\end{rem}
	From now on we assume that $\mathbb{G}$ is connected. It is equivalent to the fact that $G$ is connected, where $G$ denotes $\mathbb{G}_{ev}$, regarded as an algebraic group.
	\begin{lm}\label{invariants and Frobenius kernels}(compare with \cite[I.9.8]{jan})
	Let $char\Bbbk =p>0$. If $M$ is a $\mathbb{G}$-supermodule, then $\cap_{r\geq 1}M^{\mathbb{G}_r}=M^{\mathbb{G}}$. In particular, if $M$ and $M'$ are finite dimensional $\mathbb{G}$-supermodules, then there is $r_0\geq 1$, such that for any $r\geq r_0$ we have $M^{\mathbb{G}_r}=M^{\mathbb{G}}$ and $\mathrm{Hom}_{_{\mathbb{G}_r}\mathsf{SMod}}(M, M')=\mathrm{Hom}_{_{\mathbb{G}}\mathsf{SMod}}(M, M')$.	
	\end{lm}
	\begin{proof}
It is clear that $M^{\mathbb{G}}\subseteq \cap_{r\geq 1}M^{\mathbb{G}_r}$. Conversely, an element $m\in M$ belongs to $\cap_{r\geq 1}M^{\mathbb{G}_r}$ if and only if 	
\[\tau_M(m)-m\otimes 1\in\cap_{r\geq 1} M\otimes \Bbbk[\mathbb{G}]\mathfrak{M}_0^{p^r}= \cap_{r\geq 1}M\otimes \Bbbk[\mathbb{G}]\mathfrak{M}^{p^r}=M\otimes \cap_{r\geq 1}\mathfrak{M}^{p^r},\]	
where $\mathfrak{M}=\ker\epsilon_{\mathbb{G}}$. Since $\mathbb{G}$ is connected, \cite[Proposition 3.4]{gr-zub} infers $\cap_{r\geq 1}\mathfrak{M}^{p^r}=0$. The second and third statements are now obvious.
	\end{proof}

If $L$ is a simple $\mathbb{G}$-supermodule, then its \emph{multiplicity} in a $\mathbb{G}$-supermodule $M$ is defined as follows.  Let
$M=\varinjlim M_i$, where $\{M_i\}_{i\in I}$ is a direct system of finite dimensional $\mathbb{G}$-supersubmodules of $M$, then $[M : L]=\sup_{i\in I} [M_i : L]$. 

Assume that the isomorphism classes of simple $\mathbb{G}$-supermodules are indexed by an interval finite poset $\Lambda$, say $\{L(\lambda)\mid\lambda\in\Lambda \}$. Let $M$ be a $\mathbb{G}$-supermodule and $\Omega\subseteq\Lambda$.
Then there is the largest $\mathbb{G}$-supersubmodule $N$ of $M$, such that $[M : L(\lambda)]\neq 0$ implies $\lambda\in\Omega$. We denote this supersubmodule by $O_{\Omega}(M)$.
For any $\mathbb{G}$-supermodule $M$ we determine $W(M)=\{\lambda\in\Lambda| [M : L(\lambda)]\neq 0\}$.

A $\mathbb{G}$-supermodule $M$ is said to be \emph{weakly restricted}, if for any finite dimensional $\mathbb{G}$-supermodule $N$, $\dim \mathrm{Hom}_{_{\mathbb{G}}\mathsf{SMod}}(N, M)
<\infty$. Note that an injective hull $I(\lambda)$ of the simple supermodule $L(\lambda)$ is weakly restricted, since 
\[\dim \mathrm{Hom}_{_{\mathbb{G}}\mathsf{SMod}}(N, I(\lambda))=[N : L(\lambda)]+[N : \Pi L(\lambda)]<\infty.\]
Further, a $\mathbb{G}$-supermodule $M$ is said to be \emph{restricted}, if for any finitely generated ideal $\Omega$ in $\Lambda$ there is $\dim O_{\Omega}(M)<\infty$. 
The second property obviously implies the first one.  Indeed, it is easy to see that $\mathrm{Hom}_{_{\mathbb{G}}\mathsf{SMod}}(N, M)=\mathrm{Hom}_{_{\mathbb{G}}\mathsf{SMod}}(N, O_{\Omega_N}(M))$, where $\Omega_N$ is the ideal generated by $W(N)$.

From now on we assume that $\mathfrak{g}_1=\mathfrak{g}_-\oplus\mathfrak{g}_+$, where
both $\mathfrak{g}_-$ and $\mathfrak{g}_+$ are $G$-invariant abelian supersubalgebras. Then $(G, \mathfrak{g}_-)$ and $(G, \mathfrak{g}_+)$ are Harish-Chandra subpairs of $(G, \mathfrak{g}_1)$. The corresponding supersubgroups are denoted by $\mathbb{P}_-$ and $\mathbb{P}_+$ respectively. They are so-called \emph{distinguished parabolic} supersubgroups
(see \cite[Section 5]{taiki}). 

Note that both $\mathbb{P}_-$ and $\mathbb{P}_+$ are split, so that $\mathbb{P}_{\pm}\simeq \mathbb{G}_{ev}\ltimes \mathbb{U}_{\pm}$, where $\mathbb{U}_{\pm}$ are purely odd normal supersubgroups of $\mathbb{P}^{\pm}$, those are corresponding to the Harish-Chandra subpairs $(1, \mathfrak{g}_{\pm})$. 

	\begin{tr}\label{char free Theorem 10.4}(compare with \cite[Theorem 10.4]{DSfunctor})
If $char\Bbbk>0$, then a weakly restricted $\mathbb{G}$-supermodule $M$ is injective if and only if the following conditions hold : 
\begin{enumerate}
	\item For any $r\geq 1$ the $\mathbb{G}_r$-supermodule $M|_{\mathbb{G}_r}$ is a direct sum of finite dimensional $\mathbb{G}_r$-supersubmodules;
	\item $M|_{\mathbb{P}_-}$ and $M|_{\mathbb{P}_+}$ are injective.
\end{enumerate}
If $char\Bbbk=0$ and $G$ is linearly reductive, then $M$ is injective if and only if it is a direct sum of finite dimensional $\mathbb{G}$-supersubmodules $M_i, i\in I$, each of which
satisfies $X_{M_i}=0$.
	\end{tr}
	\begin{proof}
Part "only if". The sheaf quotients $\mathbb{G}/\mathbb{P}^{\pm}\simeq \mathbb{U}^{\mp}$ are affine, hence (2) follows. Similarly, since each Frobenius kernel $\mathbb{G}_r$ is a normal supersubgroup, $M|_{\mathbb{G}_r}$ is injective. Furthermore, each $\mathbb{G}_r$, being an infinitesimal supergroup, has an integral (cf. \cite{zub-fer}). Therefore, Proposition \ref{equivalent conditions to existence of projectives}(4) infers that $M|_{\mathbb{G}_r}$ is a direct sum of finite dimensional indecomposable injective supermodules. 
Analogously, the condition "only if" in the case $char\Bbbk=0$ follows by Proposition \ref{equivalent conditions to existence of projectives}(4) and  Proposition \ref{injectives and DS-functor}.

Part "if". Let $char\Bbbk >0$. Recall that if $V$ is a finite dimensional $\mathbb{G}$-supermodule, then $\mathrm{Hom}_{_{\mathbb{G}}\mathsf{SMod}}(V, M)$ is a finite dimensional superspace.    
Further,  $\varprojlim$ is exact on the inverse systems of finite dimensional vector (super)spaces (cf. \cite[Proposition 7]{roos}), hence all we need is to prove that for any embedding $V\to W$ of finite dimensional $\mathbb{G}$-supermodules, the induced superspace morphism 
$\mathrm{Hom}_{_{\mathbb{G}}\mathsf{SMod}}(W, M)\to \mathrm{Hom}_{_{\mathbb{G}}\mathsf{SMod}}(V, M)$ is surjective. 

By Lemma \ref{invariants and Frobenius kernels} it remains to show that 
$\mathrm{Hom}_{_{\mathbb{G}_r}\mathsf{SMod}}(W, M)\to \mathrm{Hom}_{_{\mathbb{G}_r}\mathsf{SMod}}(V, M)$ is surjective for sufficiently large positive integer $r$. On the other hand, $M|_{\mathbb{G}_r}$ is a direct sum
of finite dimensional supersubmodules. Moreover, since $M|_{\mathbb{P}^{+}_r}$ and $M|_{\mathbb{P}^{-}_r}$ are injective, each of these summands is injective being restricted to both $\mathbb{P}^{+}_r$ and $\mathbb{P}^{-}_r$. Summing all up, one has to show that a finite dimensional $\mathbb{G}_r$-supermodule $M$ is injective, or equivalently, projective, provided
$M|_{\mathbb{P}^{+}_r}$ and $M|_{\mathbb{P}^{-}_r}$ are injective/projective.  Due to this remark, we work with finite dimensional (projective) supermodules only. 

Using \cite[Lemma 14]{fer-zub}, one can reformulate Lemma \ref{converse of the above prop} for a projective $\mathbb{P}^{\pm}_r$-supermodule $M$ as follows : \[M\simeq\mathrm{coind}^{\mathbb{P}^{\pm}_r}_{\mathbb{G}_{ev, r}} M_{\mathrm{Dist}(U^{\pm}_r)}=\mathrm{Dist}(\mathbb{P}^{\pm}_r)\otimes_{\mathrm{Dist}(\mathbb{G}_{ev, r})} M_{\mathrm{Dist}(U^{\pm}_r)},\] where $M_{\pm}=M_{\mathrm{Dist}(U^{\pm}_r)}=M/\mathrm{Dist}(U^{\pm}_r)^+ M$ are projective $\mathbb{G}_{ev, r}$-(super)modules. 
In particular, each $M_{\pm}$ can be identified with a direct summand of the $\mathbb{G}_{ev, r}$-supermodule $M$ and the above isomorphisms
are induced by the splitting morphisms $M_{\pm}\to M$. Further, for any couple of (finite dimensional) $\mathbb{G}_r$-supermodules $M$ and $N$, such that
$M|_{\mathbb{P}_r^+}$ and $N|_{\mathbb{P}_r^-}$ are projective, we have the induced morphism $\phi : \mathrm{coind}^{\mathbb{G}_r}_{\mathbb{G}_{ev, r}} M_{+}\otimes N_{-}\to M\otimes N$ of $\mathbb{G}_r$-supermodules.  

Indeed, recall that $\mathrm{Dist}(\mathbb{G}_r)\simeq \mathrm{Dist}(\mathbb{U}^+_r)\otimes \mathrm{Dist}(\mathbb{U}^-_r)\otimes \mathrm{Dist}(\mathbb{G}_{ev, r})$, and this isomorphism is induced by the product of the superalgebra $\mathrm{Dist}(\mathbb{G}_r)$. Then
\[\phi(u^+\otimes u^-\otimes m\otimes n)=\sum (-1)^{|u^-_{(2)}||m|+|u^+_{(2)}|(|m|+|u^-_{(1)}|)} u^+_{(1)}u^-_{(1)}m\otimes u^+_{(2)}u^-_{(2)}n,  \]	
\[u^+\in\mathrm{Dist}(U^+_r), \ u^-\in\mathrm{Dist}(U^-_r), \ m\in M_{+}, \ n\in N_{-}.\]
Note that $u^-_{(2)}n=\epsilon_{\mathrm{Dist}(\mathbb{U}^-_r)}(u^-_{(2)})n$, hence
\[\phi(u^+\otimes u^-\otimes m\otimes n)=\sum (-1)^{|u^+_{(2)}|(|m|+|u^-|)} u^+_{(1)}u^-m\otimes u^+_{(2)}n.\]
Recall that $\mathrm{Dist}(\mathbb{U}^{\pm}_r)\simeq\Lambda(\mathfrak{g}_{\pm})$. Thus $M\otimes N$ has a natural superspace filtration :
\[(M\otimes N)_t=\sum_{0\leq k+s\leq t} \Lambda^k(\mathfrak{g}_-)M_+\otimes \Lambda^s(\mathfrak{g}_+)M_-, \ 0\leq t\leq \dim\mathfrak{g}_- +\dim\mathfrak{g}_+.\] 
Without loss of generality, one can assume that $u^-\in \Lambda^k(\mathfrak{g}_-)\setminus 0$ and $u^+\in \Lambda^s(\mathfrak{g}_+)\setminus 0$.  Since $[\mathfrak{g}_+, \mathfrak{g}_-]\subseteq\mathfrak{g}_0$ and $[\mathfrak{g}_0, \mathfrak{g}_{\pm}]\subseteq\mathfrak{g}_{\pm}$, the induction on $k$ shows that 
\[\phi(u^+\otimes u^-\otimes m\otimes n)\equiv u^- m\otimes u^+ n \pmod{(M\otimes N)_{k+s-1}},\]
that in turn implies that $\phi$ is surjective. Comparing dimensions, one concludes that $\phi$ is an isomorphism, whence $M\otimes N$ is a projective $\mathbb{G}_r$-supermodule.
It remains to set $N=M^*$ and use the final argument from \cite[Theorem 10.4]{DSfunctor}. 

If $char\Bbbk=0$, then the full subcategory of finite diemnsional $\mathbb{G}$-supermodules is equivalent to the category $\mathfrak{F}(\mathfrak{g})$ from \cite{DSfunctor} and we can refer directly to \cite[Theorem 10.4]{DSfunctor}.  
\end{proof}
In conclusion, we will show that if $\mathbb{G}$ is quasi-reductive, then any finite direct sum of indecomposable injective $\mathbb{G}$-supermodules is restricted. 

First, recall some standard facts from the representation theory of $G$. Fix a maximal torus $T\leq G$ and choose a Borel subgroup $B\leq G$ containing $T$. Then $B=T\ltimes U$, where $U=B_u$ is the unipotent radical of $B$. 

Let $X(T)$ denote the \emph{character group} of $T$. The choice of the
Borel subgroup $B$ determines a system of positive roots $R^+$ of the set of roots $R\subseteq X(T)$. Thus we denote $B$ by $B^+$ and $U$ by $U^+$ respectively. The subset $R^+$ determines a partial order on
$X(T)$ by the rule $\mu\leq_{ev}\lambda$ if $\lambda-\mu\in \mathbb{Z}_{\geq 0}R^+$ . The system $R^- =-R^+$ of negative roots corresponds to the \emph{opposite} Borel subgroup $B^-$. 

Each $\lambda\in X(T)$ corresponds to the one dimensional $B^-$-module $\Bbbk_{\lambda}$ and in turn, to the induced $G$-module $ind^G_{B^-}\Bbbk_{\lambda}\simeq \mathrm{H}^0(G/B^{-}, \Bbbk_{\lambda})$. It is also denoted by $\mathrm{H}_{ev}^0(\lambda)$ and if it is nonzero, then it is said to be a \emph{costandard module} of the \emph{highest weight} $\lambda$. In the latter case $\lambda$ is also called 
\emph{dominant}. The subset of dominant weights is denoted by $X(T)^+$. 
\begin{pr}\label{well known facts}(cf. \cite[Proposition II.2.2 and Corollary II.2.3]{jan})
	Let $\lambda$ be a dominant weight. Then the following hold :
	\begin{enumerate}
		\item If $\mathrm{H}_{ev}^0(\lambda)_{\mu}\neq 0$, then $\mu\leq_{ev}\lambda$ and $\mathrm{H}_{ev}^0(\lambda)_{\lambda}=\mathrm{H}_{ev}^0(\lambda)^{U^+}$ is one dimensional;
		\item The socle of $\mathrm{H}_{ev}^0(\lambda)$ is a simple $G$-module, denoted by $L_{ev}(\lambda)$;
		\item Any simple $G$-module is isomorphic to $L_{ev}(\lambda)$ for some $\lambda\in X(T)^+$;
		\item Any composition quotient of $\mathrm{H}_{ev}^0(\lambda)$, different from its socle, has a form $L_{ev}(\mu)$ with $\mu <_{ev}\lambda$; 
	\end{enumerate}	
\end{pr}
Now we can build a fragment of representation theory of $\mathbb{G}$. Note that $(B^+, \mathfrak{g}_+)$ is a Harish-Chandra subpair of $(G, \mathfrak{g}_1)$, that determines a Borel
supersubgroup $\mathbb{B}^+$. Symmetrically, we have the opposite Borel supersubgroup $\mathbb{B}^-$. The root system $R$ can be extended to the root system $\Delta$, such that
\[\mathfrak{g}_-=\oplus_{\gamma\in\Delta^-_1}(\mathfrak{g}_-)_{\gamma}, \ \mathfrak{g}_0=\oplus_{\gamma\in R} (\mathfrak{g}_0)_{\gamma}, \ \mathfrak{g}_+=\oplus_{\gamma\in\Delta^+_1}(\mathfrak{g}_+)_{\gamma}, \]
where $\Delta_0=\Delta_0^-\sqcup\Delta_0^+, \Delta_0^{\pm}=R^{\pm}$. The aforementioned partial order can be extended as $\mu\leq\lambda$ if $\lambda-\mu\in\mathbb{Z}_{\geq 0}\Delta^+$.
Set $\mathrm{H}^0(\lambda)=\mathrm{ind}^{\mathbb{G}}_{\mathbb{B}^-}\Bbbk_{\lambda}$.
\begin{pr}\label{well known superfacts}(cf. \cite[Proposition 4.11, Proposition 4.15 and Proposition 5.4]{taiki})
	\begin{enumerate}
		\item $\mathrm{H}^0(\lambda)\neq 0$ if and only if $\mathrm{H}_{ev}^0(\lambda)\neq 0$;
	\item If $\mathrm{H}^0(\lambda)_{\mu}\neq 0$, then $\mu\leq\lambda$ and $\mathrm{H}^0(\lambda)_{\lambda}=\mathrm{H}^0(\lambda)^{\mathbb{U}^+}$ is one dimensional;
	\item The socle of $\mathrm{H}^0(\lambda)$ is a simple $G$-module, denoted by $L(\lambda)$; 
	\item Any simple $G$-module is isomorphic to $L(\lambda)$ or $\Pi L(\lambda)$ for some $\lambda\in X(T)^+$;
	\item Any composition quotient of $\mathrm{H}^0(\lambda)$, different from its socle,  has a form $L(\mu)$ or $\Pi L(\mu)$, with
	$\mu <\lambda$. 
\end{enumerate}		
\end{pr}
Proposition \ref{well known superfacts}(1) shows that the simple $\mathbb{G}$-supermodules are indexed by the poset $\Lambda=X(T)^+\times\mathbb{Z}_2$, where $\Pi^a L(\lambda)\leftrightarrow (\lambda, a)$ and $(\mu, a)\leq (\lambda, b)$ if $\mu\leq \lambda$. It is well known that $X(T)^+$ is interval finite with respect to $\leq_{ev}$. But it is not at all obvious that $\Lambda$ is. In what follows we just assume that it is the case. 
\begin{lm}\label{induced via P}
Let $\lambda$ be a diminant weight. Then $\mathrm{H}^0(\lambda)\simeq \mathrm{ind}^{\mathbb{G}}_{\mathbb{P}^-} \mathrm{H}_{ev}^0(\lambda)$. 	
\end{lm}
\begin{proof}
Since $\mathbb{P}^-\simeq G\ltimes\mathbb{U}^-$ and $\mathbb{B}^-\simeq B^-\ltimes\mathbb{U}^-$, \cite[Lemma 10.4]{zub2} infers 
$\mathrm{ind}^{\mathbb{P}^-}_{\mathbb{B}^-}\Bbbk_{\lambda}\simeq \mathrm{H}_{ev}^0(\lambda)$, where the latter module is regarded as a $\mathbb{P}^-$-supermodule via
the natural supergroup morphism $\mathbb{P}^-\to G$. Our statement follows by the transitivity of the induction functor. 	
\end{proof}
Recall that $\mathrm{V}_{ev}(\lambda)=\mathrm{H}_{ev}^0(-w_0\lambda)^*$ is called the \emph{universal highest weight module} of weight $\lambda$, or \emph{Weyl module} of highest weight $\lambda$ (cf. \cite[II.2.13]{jan}). Set $\mathrm{V}(\lambda)=\mathrm{coind}^{\mathbb{G}}_{\mathbb{P}^+} \mathrm{V}_{ev}(\lambda)$ and call it the \emph{Weyl or standard supermodule}
of highest weight $\lambda$. 
\begin{pr}\label{properties of Weyl}
	We have 
\begin{enumerate}
	\item The top of $\mathrm{V}(\lambda)$ is isomorphic to $L(\lambda)$;
	\item Any composition quotient of the radical of $\mathrm{V}(\lambda)$ has a form $L(\mu)$ or $\Pi L(\mu)$, with $\mu<\lambda$;
	\item If $M$ is a finite dimensional $\mathbb{G}$-supermodule, that satisfies the conditions (1) and (2), then there is a surjective $\mathbb{G}$-supermodule morphism 
	$\mathrm{V}(\lambda)\to M$. 
\end{enumerate}	
\end{pr}
\begin{proof}
Let $\Pi^a L(\mu), a=0, 1,$ be a summand of the top of $\mathrm{V}(\lambda)$. 
By (co)Frobenius reciprocity, we have
\[\mathrm{Hom}_{_{\mathbb{G}}\mathsf{Smod}}(\mathrm{V}(\lambda), \Pi^a L(\mu))\simeq \mathrm{Hom}_{_{\mathbb{P}^+}\mathsf{Smod}}(\mathrm{V}_{ev}(\lambda), \Pi^a L(\mu))\neq 0,\]
and \cite[Lemma II.2.13(a)]{jan}, combined with Proposition \ref{well known superfacts}(2), infers that $\lambda=\mu$ and $a=0$. The second statement 
is obvious due the isomorphism of $T$-supermodules $\mathrm{V}(\lambda)\simeq \Lambda(\mathfrak{g}_-)\otimes \mathrm{V}_{ev}(\lambda)$.	

To prove (3) we note that by (superized) Nakayama's lemma, $M$ is generated by a $\mathbb{B}^+$-stable vector of weight $\lambda$. Then use again (co)Frobenius reciprocity
and \cite[Lemma II.2.13(b)]{jan}.  
\end{proof} 
\begin{lm}\label{Ext is trivial}
For any dominant weights $\lambda$ and $\mu$, we have
\begin{enumerate}
	\item $\mathrm{Ext}^n_{_{\mathbb{G}}\mathsf{SMod}}(\mathrm{V}(\lambda), \mathrm{H}^0(\mu))=0$, whenever $n> 0$;
	\item $\mathrm{Ext}_{_{\mathbb{G}}\mathsf{SMod}}^0(\mathrm{V}(\lambda), \mathrm{H}^0(\mu))=\mathrm{Hom}_{_{\mathbb{G}}\mathsf{SMod}}(\mathrm{V}(\lambda), \mathrm{H}^0(\mu))\neq 0$ if and only if $\lambda=\mu$. Moreover, in the latter case
$\mathrm{Hom}_{_{\mathbb{G}}\mathsf{SMod}}(\mathrm{V}(\lambda), \mathrm{H}^0(\lambda))$ is the even one dimensional superspace. 	 
\end{enumerate}	
\end{lm}	
\begin{proof}
Arguing as in \cite[Proposition 8.6]{zub3}, we obtain 
\[\mathrm{Ext}^n_{_{\mathbb{G}}\mathsf{SMod}}(\mathrm{V}(\lambda), \mathrm{H}^0(\mu))\simeq \mathrm{Ext}^n_{_{\mathbb{P}^+}\mathsf{SMod}}(\mathrm{V}_{ev}(\lambda), \mathrm{H}^0(\mu)|_{\mathbb{P}^+})\simeq \mathrm{Ext}^n_{_{G}\mathsf{mod}}(\mathrm{V}_{ev}(\lambda), \mathrm{H}^0(\mu)^{\mathbb{U}^+}).\]	
By \cite[Lemma 8.2 and Remark 8.3]{zub2}, $\mathrm{H}^0(\mu)\simeq \mathrm{H}_{ev}^0(\mu)\otimes \Bbbk[\mathbb{U}^+]$ and $\mathrm{H}^0(\mu)^{\mathbb{U}^+}\simeq \mathrm{H}_{ev}^0(\mu)$
(as a $G$-module). Then both statements follow by \cite[Proposition II.4.13]{jan}. 
\end{proof}
Following \cite{jan}, we say that a $\mathbb{G}$-supermodule $M$ has a \emph{good filtration} if $M$ has an ascending chain
\[0=M(0)\subseteq M(1)\subseteq \ldots M(i)\subseteq \ldots\]
of $\mathbb{G}$-supersubmodules, such that $\cup_{i\geq 0} M_i=M$ and each $M_i/M_{i-1}$ is isomorphic to some $\Pi^a \mathrm{H}^0(\lambda_i), (\lambda_i, a)\in\Lambda$. 
The standard long exact sequence arguments and Proposition \ref{Ext is trivial} imply that $\mathrm{Ext}^n_{_{\mathbb{G}}\mathsf{SMod}}(\Pi^a \mathrm{V}(\lambda), M)=0$, provided $n> 0$ and $M$ has a good filtration. The converse is also true for so-called \emph{bounded} supermodules (cf. \cite{zub5}). More precisely, a $\mathbb{G}$-supermodule is called \emph{bounded}, if
$\dim \mathrm{Hom}_{_{\mathbb{G}}\mathsf{SMod}}(\Pi^a\mathrm{V}(\pi), M)<\infty$ for any $(\pi, a)\in\Lambda$ and the poset 
\[\hat{M}=\{(\pi, a)\in\Lambda\mid \dim \mathrm{Hom}_{_{\mathbb{G}}\mathsf{SMod}}(\Pi^a\mathrm{V}(\pi), M)\neq 0\}=\]\[\{\pi\in X(T)^+\mid \dim \mathrm{Hom}_{_{\mathbb{G}}\mathsf{SMod}}(\mathrm{V}(\pi), M)\neq 0\}\times\mathbb{Z}_2\]
does not contain any infinite descending chains.
\begin{tr}\label{good filtration}(see \cite{jan, zub5} for more details)
A bounded $\mathbb{G}$-supermodule $M$ has a good filtration if and only if $\mathrm{Ext}^1_{_{\mathbb{G}}\mathsf{SMod}}(\mathrm{V}(\lambda), M)=0$ for any dominant weight $\lambda$.	
\end{tr}
\begin{proof}
Choose a minimal element $(\lambda, a)\in \hat{M}$. Arguing as in \cite[Lemma II.4.15]{jan}, one can derive that $M$ contains a supersubmodule that is isomorphic to $\Pi^a\mathrm{H}(\lambda)$. Since $\mathrm{Hom}_{_{\mathbb{G}}\mathsf{SMod}}(\mathrm{V}(\lambda), \ )$  is exact on supermodules with good filtration, we have
\[\dim \mathrm{Hom}_{_{\mathbb{G}}\mathsf{SMod}}(\mathrm{V}(\lambda), M/\Pi^a\mathrm{H}(\lambda))= \dim \mathrm{Hom}_{_{\mathbb{G}}\mathsf{SMod}}(\mathrm{V}(\lambda), M)-1,\]	
and the standard exhaustion arguments conclude the proof. 
\end{proof}
\begin{cor}\label{injectives are bounded}
The injective hull $I(\lambda)$ of the simple supermodule $L(\lambda)$ is bounded, hence it has a good filtration. In particular, it is restricted. 	
\end{cor}
\begin{proof}
We have $\dim\mathrm{Hom}_{\mathbb{G}}(\mathrm{V}(\pi), I(\lambda))=[\mathrm{V}(\pi) : L(\lambda)]+[\mathrm{V}(\pi) : \Pi L(\lambda)]<\infty$. Moreover, if
$\mathrm{Hom}_{_{\mathbb{G}}\mathsf{SMod}}(\mathrm{V}(\pi), I(\lambda))\neq 0$, then $\pi\geq \lambda$ any descending chain in $\hat{I(\lambda)}$, starting from $\pi$, is contained in the finite interval 
$[\lambda, \pi]$.  	

Observe that $\lambda$ is the smallest weight in $\hat{I(\lambda)}$. Let $\Omega$ be an ideal in $\Lambda$, finitely generated by the elements $\pi_1, \ldots, \pi_t$. Let  $M=O_{\Omega}(I(\lambda))\neq 0$. Fix a good filtration
$\{I(\lambda)_i\}_{i\geq 0}$ of $I(\lambda)$. Set $M_i=M\cap I(\lambda)_i, i\geq 0$. Since $M/M_i=O_{\Omega}(I(\lambda)/M_i)$ and $M_i/M_{i-1}\subseteq \Pi^{a_i}\mathrm{H}^0(\lambda_i)$, $M_i/M_{i-1}\neq 0$ if and only if $M_i/M_{i-1}=\Pi^{a_i}\mathrm{H}^0(\lambda_i)$ if and only if $\lambda_i$ belongs to $\Omega$. Thus \[\dim O_{\Omega}(I(\lambda))= \sum_{\lambda_i\in\Omega}\dim \mathrm{H}^0(\lambda_i)\leq \sum_{\lambda_i\in\cup_{1\leq k\leq t}[\lambda, \pi_k]}\dim \mathrm{H}^0(\lambda_i)<\infty .\] 
\end{proof}
\begin{cor}\label{universality of H^0}
If $M$ is a finite dimensional $\mathbb{G}$-supermodule, that satisfies the conditions (3) and (5) of Proposition \ref{well known superfacts}, then there is an
injective $\mathbb{G}$-supermodule morphism $M\to \mathrm{H}^0(\lambda)$.	
\end{cor}
\begin{proof}
The supermodule $M$ can be embedded into $I(\lambda)$, where $L(\lambda)$ is the socle of $M$. It remains to note that $O_{\Omega}(M)=M$ and $O_{\Omega}(I(\lambda))=\mathrm{H}^0(\lambda)$, where
$\Omega$ is the ideal generated by $\lambda$.	
\end{proof}
The first statement of Theorem \ref{good filtration} and Corollary \ref{universality of H^0} say that the category of $\mathbb{G}$-supermodules is a \emph{highest weight } category in the sense of \cite{CPS}.

For any root $\alpha\in\Delta$, let $\mathbb{U}_{\alpha}$ denote the root unipotent (super)subgroup, such that $\mathrm{Lie}(\mathbb{U}_{\alpha})=\mathfrak{g}_{\alpha}$.  If $\alpha\in\Delta_0$, then set $\mathbb{U}_{\alpha, r}=\mathbb{U}_{\alpha}\cap\mathbb{G}_r, r\geq 1$. Due \cite[Main Theorem]{CPS2} the first statement of Theorem \ref{char free Theorem 10.4} can be reformulated as follows.
\begin{cor}\label{one more about Thorem 10.4}
A weakly restricted $\mathbb{G}$-supermodule $M$ is injective if and only if the following conditions hold :
\begin{enumerate}
	\item For any $r\geq 1$ the $\mathbb{G}_r$-supermodule $M|_{\mathbb{G}_r}$ is a direct sum of finite dimensional $\mathbb{G}_r$-supersubmodules;
	\item $M|_{\mathbb{U}_{\alpha}}$ is injective for any $\alpha\in\Delta_1$;
	\item $M^{\mathbb{U}^{\pm}}|_{U_{\alpha, r}}$ is injective for any $\alpha\in\Delta_0$ and $r\geq 1$.
\end{enumerate}	
\end{cor}
\begin{proof}
If $\alpha\in\Delta_1$, then $\mathbb{U}_{\alpha}\leq\mathbb{G}_r$ and (1) and (2)  (combined with Proposition \ref{if G is purely odd}) imply that
$M|_{\mathbb{U}^{\pm}}$ are injective. Similarly, (1), (3) and \cite[Main Theorem]{CPS2} imply that both $M^{\mathbb{U}^{\pm}}$ are injective $TG_r$-modules, hence 
injective $G_r$-modules. By Lemma \ref{converse of the above prop}, $M|_{\mathbb{P}^{\pm}_r}$ are injective.  
\end{proof}
The typical examples of quasi-reductive supergroups, to which we can apply Theorem \ref{char free Theorem 10.4}, are $\mathrm{GL}(m|n), \mathrm{SL}(m|n), m\neq n$, and $\mathrm{P}(n)$. 
Due Corollary \ref{injectives are bounded}, all we need is to show that for suitable choice of the system of positive roots, the induced partial order on the corresponding set of dominant weights is interval finite. 
   
The first case was considered in detail in \cite{zub5}. The case of supergroup $\mathrm{SL}(m|n)$ is close to the first one and we leave it for the reader.

Let $\mathbb{G}=\mathrm{P}(n)\leq\mathrm{GL}(n|n)$. Then $G\simeq \mathrm{GL}(n)$ and we fix a maximal torus $T$ consisting of all diagonal matrices from $\mathrm{P}(n)$. 
The distinguished parabolic supersubgroup $\mathbb{P}^-$ is described in \cite[Example 5.10]{taiki}. The opposite distinguished parabolic supersubgroup $\mathbb{P}^+$ consists of all matrices
\[\left(\begin{array}{cc}
	X & 0 \\
	Z & (X^t)^{-1} 
\end{array}\right), X\in\mathrm{GL}(n), Z\in\mathrm{M}(n), X^t Z=-Z^t X.\]
The root system of $\mathrm{P}(n)$ is $\Delta=\Delta_0\sqcup\Delta_1$, where
\[\Delta_0=\{\pm(\epsilon_i-\epsilon_j)| 1\leq i< j\leq n \},\]
\[\Delta_1=\{\pm(\epsilon_i+\epsilon_j), 2\epsilon_k | 1\leq i< j\leq n, 1\leq k\leq n\},\]
and $\epsilon_1, \ldots, \epsilon_n$ is the standard basis of $X(T)\simeq\mathbb{Z}^n$. 
We also have
\[\Delta^+_0=\{\epsilon_i-\epsilon_j | 1\leq i< j\leq n\} \ \mbox{and} \ \Delta^+_1=\{-(\epsilon_i+\epsilon_j) | 1\leq i< j\leq n\}.\]
Note that $\Delta^-\neq -\Delta^+$. 
\begin{pr}
The set of dominant weights $\Lambda$ is an interval finite poset with respect to the order $\leq$ determined by $\Delta^+$. In particular, the category $_{\mathrm{P}(n)}\mathsf{Smod}$ is a highest weight one. 	
\end{pr}
\begin{proof}	
	For any character $\lambda=\sum_{1\leq i\leq n}a_i\epsilon_i$, let $|\lambda|$ denote the \emph{norm} $\sum_{1\leq i\leq n}a_i$ of $\lambda$.
If $\mu\leq\lambda$, then $l(\mu, \lambda)=\frac{|\mu|-|\lambda|}{2}\geq 0$ is the number of roots from  $\Delta^+_1$ appearing in $\lambda-\mu$. Besides, $l(\mu, \pi)+l(\pi, \lambda)=l(\mu, \lambda)$ for any $\pi\in [\mu, \lambda]$. 

The additive semigroup $\mathbb{Z}_{\geq 0}\Delta^+$ is generated by the roots $\epsilon_i-\epsilon_{i+1}, -(\epsilon_1+\epsilon_2), 1\leq i\leq n$. Therefore, if $\mu, \pi$ and $\lambda$ are dominant weights and $\pi\in [\mu, \lambda]$, then \[a_1+l(\mu, \lambda)\geq a_1+l(\pi, \lambda)\geq b_1\geq\ldots\geq b_n\geq a_n, \]
where $\lambda=\sum_{1\leq i\leq n} a_i\epsilon_i, \pi=\sum_{1\leq i\leq n} b_i\epsilon_i$.
\end{proof}

\end{document}